\documentclass[a4paper]{amsart}
\usepackage{amssymb,amsfonts}
\usepackage[all,arc]{xy}
\usepackage{enumerate}
\usepackage{mathrsfs}
\usepackage{tikz}
\usetikzlibrary{arrows,snakes,backgrounds}
\usepackage{color}   %May be necessary if you want to color links
\usepackage{marginnote}
\usepackage{tikz-cd}
\usetikzlibrary{positioning}
\usepackage{enumitem}
\usepackage{hyperref}
\usepackage{MnSymbol} 
\hypersetup{
    colorlinks=true, %set true if you want colored links
    linktoc=all,     
    citecolor=black,
    filecolor=black,
    linkcolor=blue,
    urlcolor=black
}
\usepackage{cleveref}% Has to be loaded after hyperref
\usetikzlibrary{arrows}

\newtheorem{thm}{Theorem}[section]
\newtheorem{cor}[thm]{Corollary}
\newtheorem{prop}[thm]{Proposition}
\newtheorem{lem}[thm]{Lemma}

\newtheorem{quest}[thm]{Question}

\newtheorem*{openproblem*}{Problem}
\newtheorem*{quest*}{Question}
\newtheorem*{problem*}{Problem}

\theoremstyle{definition}
\newtheorem{defn}[thm]{Definition}

\theoremstyle{remark}
\newtheorem{rem}[thm]{Remark}

\newcommand{\bR}{\mathbb{R}}

\newcommand{\bZ}{\mathbb{Z}}

\newcommand\Diff{\mathrm{Diff}}

\newcommand{\hcoker}{/\!\!/}

\newcommand{\tH}{\text{\textnormal{Homeo}}}
\newcommand{\BH}{\mathrm{B}\text{\textnormal{Homeo}}}

\newcommand{\tdH}{\text{Homeo}^{\delta}}

\newcommand{\BdH}{\mathrm{B}\text{\textnormal{Homeo}}^{\delta}}

\addtocontents{toc}{\protect\setcounter{tocdepth}{1}}
\makeatletter
\let\c@equation\c@thm
\makeatother
\numberwithin{equation}{section}
\title{A local to global argument on low dimensional manifolds}

\author{Sam Nariman}
\email{sam@math.northwestern.edu}
\address{Department of Mathematics\\
  Northwestern University\\
2033 Sheridan Road\\
Evanston, IL  60208}
\begin{document}

\begin{abstract}
For   an oriented manifold $M$ whose dimension is less than $4$, we use the contractibility of certain complexes associated to its submanifolds to cut $M$ into simpler pieces in order to do  local to global arguments. In particular, in these dimensions, we give a different proof of a deep theorem of Thurston in foliation theory  that says the natural map between classifying spaces $\BdH(M)\to \BH(M)$ induces a homology isomorphism where $\tdH(M)$ denotes the group of homeomorphisms of $M$ made discrete. Our proof shows that in low dimensions,  Thurston's theorem can be proved without using foliation theory. Finally, we show that this technique gives a new perspective on the homotopy type of homeomorphism groups in low dimensions. In particular, we give a different proof of Hacher's theorem that the homeomorphism groups of Haken $3$-manifolds with boundary are homotopically discrete without using his disjunction techniques.
\end{abstract}
\maketitle
\section{Introduction}
Often, in h-principle type theorems (e.g. Smale-Hirsch theory), it is easy to check that the statement holds for the open disks (local data) and then one wishes to glue them together to prove that the statement holds for closed compact manifolds (global statement). But there are cases where one has a local statement for a closed disk relative to the boundary. To use such local data to great effect, instead of covering the manifold by open disks, we use certain ``resolutions" associated to submanifolds (see \Cref{sec2}) to cut the manifold into closed disks.
\subsection{Thurston's h-principle theorem for $C^0$-foliated bundles}  The main example that led us to such a local to global argument comes from foliation theory. Let $\tH(D^n,\partial D^n)$ denote the group of compactly supported homeomorphisms of the interior of the disk $D^n$ with the compact-open topology. By the Alexander trick, we know that the group $\tH(D^n,\partial D^n)$  is  contractible for all $n$. Let $\tdH(D^n,\partial D^n)$ denote  the same group as $\tH(D^n,\partial D^n)$ but with the discrete topology. By an infinite repetition trick due to   Mather (\cite{MR0288777}), it is known that $\BdH(D^n,\partial D^n)$ is acyclic i.e. its reduced homology groups vanish. Therefore, the natural map
 \[
 \BdH(D^n,\partial D^n)\to \BH(D^n,\partial D^n)
 \]
 induced by the identity homomorphism is in particular a homology isomorphism. Thurston generalized Mather's work (\cite{MR0356085}) on foliation theory in \cite{thurston1974foliations} and as a corollary he obtained  the following surprising result.
   \begin{thm}[Thurston]\label{Main}
  For a compact closed  connected manifold $M$, the map 
  \[
  \eta\colon \BdH(M)\to \BH(M),
  \]
  induces an isomorphism on homology.
  \end{thm}
In this paper, we give a proof of this theorem when $\text{dim}(M)\leq 3$. Our proof is inspired by Jekel's calculation of the group homology of $\tdH(S^1)$ in \cite[Theorem 4]{MR2871163}. The first  proof of Thurston's theorem in the literature in all dimensions was given by McDuff following Segal's program in  foliation theory (see \cite{mcduff1980homology, segal1978classifying}). Thurston in fact proved a more general homology h-principle theorem for foliations such that \Cref{Main} is just its consequence for $C^0$-foliations. 

Mather and Thurston used foliation theory to study the homotopy fiber of $\eta$. To briefly explain their point of view, let us recall the notion of Haefliger groupoid. Haefliger defined a topological groupoid $\Gamma_q^r$ whose space of objects is $\bR^q$ with the usual topology and the space of morphisms between two points is given by germs of $C^r$-diffeomorphisms sending $x$ to $y$ (see \cite[Section 1]{haefliger1971homotopy} for more details). The homotopy type of the classifying space of this groupoid, $\mathrm{B}\Gamma_q^r$, plays an important role in the classification of $C^r$-foliations (see \cite{MR0370619} and \cite{MR0425985}). One of Thurston's deep theorem in  foliation theory relates the homotopy type of $\mathrm{B}\Gamma_q^r$ to the group homology of $C^r$-diffeomorphism groups made discrete. For $r=0$, he first uses Mather's theorem (\cite{MR0288777}) to show that $\mathrm{B}\Gamma_q^0$ is weakly equivalent to the classifying space of rank $q$ microbundles, $\mathrm{B} \text{Top}(q)$, and as a consequence he deduces that the map $\eta$ in \Cref{Main} is in fact acyclic, in particular, its homotopy fiber has vanishing reduced homology groups. 
 % The purpose of this paper is to illustrate a different way of local to global argument by first giving a short proof of the above homology h-principle theorem of Thurston in all dimensions except dimension $4$.

%In  foliation theory,  the homotopy fiber of the map $\eta$ has been extensively studied. 
In fact there are general h-prinicple theorems in all dimensions that identify the homotopy fiber of $\eta$ as  a subspace of certain section spaces of a bundle associated to the manifold $M$. Our goal in this paper is to show that in low dimensions one can directly study the maps between classifying spaces like $\eta$ instead of their homotopy fibers. To do so, we provide the strategy in detail for \Cref{Main} in low dimensions that does not use any foliation theory.

\subsection{On homotopy type of $\tH_0(M)$} Using this technique, we also give a different proof of the contractibility of the identity component of homeomorphism groups in low dimensions. %First we revisit the case of hyperbolic surfaces and Haken three manifolds:
\begin{thm}[Earle-Eells-Schatz, Hatcher] The identity components of homeomorphism groups of hyperbolic surfaces (see \cite{MR0277000, earle1969fibre}) and Haken manifolds with boundary (see \cite{MR0420620}) are contractible.
\end{thm}
\begin{rem}
The same statement holds for diffeomorphism groups. In dimensions less than $4$, it is known that for a compact manifold $M$, the group  $\tH_0(M)$ with the compact-open topology and $\Diff_0(M)$ with the $C^{\infty}$-topology have the same homotopy type. In $\text{dim}(M)=2$, it is a theorem of Hamstrom \cite{hamstrom1974homotopy} and in  $\text{dim}(M)=3$ it is a deep theorem of Cerf \cite[Theorem 8]{cerf1961topologie}. 
\end{rem}
\begin{rem}
Recall that $\tH(D^n, \partial D^n)$ is contractible for all $n$ and $\Diff(D^n, \partial D^n)$ is contractible for $n\leq 3$ (see \cite{MR0112149, hatcher1983proof}). One could also use this local data and the technique of this paper to reprove Hamstrom and Cerf's theorem that $\tH_0(M)$ and $\Diff_0(M)$ have the same homotopy type.
\end{rem}
Instead of working with the homeomorphism groups, we work with their classifying spaces. Considering the delooping of these topological groups has the advantage that one can apply homological techniques to the classifying spaces to extract homotopical information about homeomorphism groups. 

The reason that we restrict ourselves to low dimensions is that for surfaces and $3$-manifolds, there is a procedure to split up  the manifold into disks. For the surfaces, this procedure is given by cutting along handles. For $3$-manifolds, however, it is more subtle to cut it into simpler pieces. In that case, we use the prime decomposition theorem and Haken's hierarchy to cut the manifold into disks. 

\subsection{Outline} The paper is organized as follows: in \Cref{sec2}, we discuss the main idea and give a different model for the map $\eta$ which will be technically more convenient. In \Cref{disk}, we discuss the case where $M$ is a surface and  semi-simplicial resolutions for the classifying spaces of homeomorphisms of surfaces. In \Cref{balls},  We will treat the case of  $3$-manifolds. In \Cref{sec4}, we give a different proof of the contractibility of the identity component of the homeomorphism groups for certain low dimensional manifolds.

\subsection*{Acknowledgment} I am indebted to Allen Hatcher for his careful  reading and many helpful comments on different versions of this paper. I first aimed to prove Thurston's theorem without using foliation theory in all dimensions but Hatcher pointed out a flaw in my argument for manifolds of dimension larger than $4$ which made me focus only on low dimensional manifolds.  I would also thank Sander Kupers for his comments on  transversality issues for topological manifolds and answering my questions about topological embeddings. Finally, I am also grateful to S\o ren Galatius for his comments on the first draft of this paper, to Kathryn Mann for her comment on sphere systems, to Elmar Vogt and Reimer Backhaus for their interest in this paper and the referee for his/her careful reading and many helpful comments and critics that improved the exposition of the paper. The author was partially supported by NSF DMS-1810644.  

\section{Resolving classifying spaces by embedded submanifolds}\label{sec2}
Let us first sketch the idea for \Cref{Main}. Let $M$ be a smooth oriented closed manifold and let $\tH_0(M)$ denote  the identity component of the topological group $\tH(M)$. Since the group $\tH(M)$ is locally path connected and in fact it is locally contractible (see \cite{MR0259925}),  the group of connected components $\pi_0(\tH(M))$ is a discrete group and sits in a short exact sequence
\[
1\to \tH_0(M)\to \tH(M)\to \pi_0(\tH(M))\to 1. 
\]
Using the Serre spectral sequence, one could reduce  \Cref{Main} to proving that  the map $$\eta\colon\BdH_0(M)\to \BH_0(M),$$ induces a homology isomorphism. To prove this version, we want to inductively reduce \Cref{Main} to the case of a simpler manifold. Such simpler manifolds are obtained from $M$  by cutting along its submanifolds. 

Let $\phi$ be an embedding of a manifold into $M$. To cut along this embedding, we construct a  semisimplicial space $A_{\bullet}(M,\phi)$ on which the topological group $\tH_0(M)$ acts (see \cite{ebert2017semi} or \cite[Section 2]{randal2009resolutions} for definitions of (augmented) semisimplicial objects and their realizations\footnote{We shall use the same notation for the realizations of semisimplicial spaces and simplicial spaces in this paper.}). Similarly we construct a semisimplicial set $A^{\delta}_{\bullet}(M,\phi)$ as  the underlying semisimplicial set of the semisimplicial space $A_{\bullet}(M,\phi)$ on which the group $\tdH_0(M)$ acts. 
These semisimplicial spaces are constructed so that their realizations are weakly contractible. Therefore, we obtain semisimplicial resolutions \footnote{For a topological group $G$ acting on a topological space $X$, the homotopy quotient is denoted by $X\hcoker G$ and is given by $X\times_G \mathrm{E}G$ where $\mathrm{E}G$ is a contractible space on which $G$ acts freely.}
\[
|A^{\delta}_{\bullet}(M,\phi)\hcoker \tdH_0(M)|\xrightarrow{\simeq}\BdH_0(M),
\]
\[
|A_{\bullet}(M,\phi)\hcoker \tdH_0(M)|\xrightarrow{\simeq}\BH_0(M).
\] We then construct a zig-zag of maps from the space $A^{\delta}_{\bullet}(M,\phi)\hcoker \tdH_0(M)$ to the space $A_{\bullet}(M,\phi)\hcoker \tH_0(M)$ which induces a  commutative diagram 
 \[
 \begin{tikzpicture}[node distance=5.6cm, auto]
 % \node (A) {$|A^{\delta}_{\bullet}(M,\phi)|$};
  \node (B)  {$H_*(|A^{\delta}_{\bullet}(M,\phi)\hcoker \tdH_0(M)|;\bZ)$};
  \node (C) [below of= B, node distance=1.6cm ] {$H_*(\BdH_0(M);\bZ)$};  
 % \node (D) [right of= A] {$|A_{\bullet}(M,\phi)|$};
    \node (E) [right of= B] {$ H_*(|A_{\bullet}(M,\phi)\hcoker \tH_0(M)|;\bZ)$};
  \node (F) [right of= C ] {$H_*(\BH_0(M);\bZ).$};  
%   \draw [->] (A) to node {$$}(D);
  \draw [->] (B) to node {$f_*$}(E);
  \draw [->] (C) to node {$\eta_*$}(F);
 % \draw [->] (A) to node {$$}(B);
  \draw [->] (B) to node {$\cong$} (C);
% \draw [->] (D) to node {$$} (E);
  \draw [->] (E) to node {$\cong$} (F);

\end{tikzpicture}
\]
 Therefore, it is enough to prove that $f_*$ is an  isomorphism.  As we shall see in \Cref{high}, proving that $f_*$ induces a homology isomorphism is equivalent to the statement of \Cref{Main} for a manifold that is obtained from $M$ by cutting it along $\phi$. Then, by induction we can reduce \Cref{Main} to the case of a disk relative to its boundary that
\[
\BdH(D^n,\partial D^n)\to \BH(D^n,\partial D^n),
\]
induces a homology isomorphism (\cite{MR0288777}).

 We restricted ourselves to the case of closed oriented manifold $M$ of dimension less than $4$, because we still do not know how to make a certain surgery argument in \Cref{arc} work in  dimensions higher than $3$. 

To prove \Cref{Main} for closed manifolds, we need to work with manifolds with boundary, we first fix two notations to deal with homeomorphisms that are relative to the boundary.
\begin{defn}For an oriented manifold $M$ with (possibly non empty) boundary, we let $\tH(M,\partial M)$ and $\tH_{\partial}(M)$ be respectively the group of compactly supported orientation preserving homeomorphisms of $\text{int}(M)$, interior of $M$,  and  the group of  orientation preserving homeomorphisms that are the identity on the boundary with the compact open topology. 
\end{defn}

It is, however, technically more convenient to work with simplicial groups to avoid subtleties of working with  the topological group $\tH(M)$.  We shall define the corresponding simplicial groups.
\begin{defn}
  The set of the $p$-simplices of the simplicial group $S_{\bullet}(\tH_{\partial}(M))$, namely the singular  complex of $\tH(M)$, can be described as the commutative diagrams
    \[
     \begin{tikzpicture}[node distance=2cm, auto]
  \node (A) {$\Delta^p\times M$};
  \node (B) [right of=A] {$$};
  \node (C) [right of=B] {$\Delta^p\times M$};  
  \node (D) [below of=B, node distance=1.3cm] {$\Delta^p,$};
  \draw[->] (C) to node {$pr_1$} (D);
  \draw [<-] (D) to node {$pr_1$} (A);
    \draw [->] (A) to node {$\phi$} (C);
\end{tikzpicture}
    \]
where $\phi$ is  a homeomorphism which is the identity on $\Delta^p\times {\partial M}$. Similarly, one does define $S_{\bullet}(\tH_{0,\partial}(M))$ and $S_{\bullet}(\tH_0(M, \partial M))$.
\end{defn}
Using the theorem of Milnor (\cite{MR0084138}), we know that the geometric realization $|S_{\bullet}(\tH_{0,\partial}(M))|$ is a topological group which is weakly equivalent to $\tH_{0,\partial}(M)$. The composite of the maps
\[
\tdH_{0,\partial}(M)\to |S_{\bullet}(\tH_{0,\partial}(M))|\xrightarrow{\text{ev}}  \tH_{0,\partial}(M),
\]
where the first one is induced by the inclusion of the $0$-simplices and the second map is induced by the evaluation map $S_{\bullet}(\tH_{0,\partial}(M))\times \Delta^{\bullet}\to \tH_{0,\partial}(M)$ is the identity homomorphism. Therefore, the map $\eta$ is factored as 
\[
\BdH_{0,\partial}(M)\to \mathrm{B} |S_{\bullet}(\tH_{0,\partial}(M))|\xrightarrow{\simeq} \BH_{0,\partial}(M).
\]
So we reformulate \Cref{Main} as follows.
\begin{thm}\label{main}
For a compact oriented smooth manifold $M$ whose dimension is less than $4$, the map
\[
\eta\colon \BdH_{0,\partial}(M)\to \mathrm{B} |S_{\bullet}(\tH_{0,\partial}(M))|,
\]
induces a homology isomorphism. 
\end{thm}
\begin{rem}
The same statement holds for $\tH_{0,\partial}(M)$. Using the pushing collar technique (\cite[Corollary 2.3]{nariman2014homologicalstability} \footnote{This corollary that says certain pushing collar maps between diffeomorphism groups induce homology isomorphisms also works for homeomorphism groups.}), one can show that the map 
\[
\BdH_{0}(M,\partial M)\to \BdH_{0,\partial}(M),
\]
induces a homology isomorphism. On the other hand, since the space of collars for topological manifolds is contractible (see \cite{MR0266221}), the inclusion map $\tH_0(M,\partial M)\hookrightarrow \tH_{0,\partial}(M)$ is a weak equivalence (see \cite[Section 4.3]{kupers2015proving} for a similar discussion). Hence, \Cref{main} also implies that the map
\[
\eta\colon \BdH_{0}(M,\partial M)\to \mathrm{B} |S_{\bullet}(\tH_{0}(M,\partial M))|,
\]
induces a homology isomorphism. 
\end{rem}
 We want to cut up $M$ into disks in a ``contractible space of choices" (e.g. see \Cref{claim} and \Cref{arc}). As we shall see in \Cref{Jekel}, the easiest case is when $M$ is homeomorphic to a circle (see also \cite[Theorem 4]{MR2871163}). For $M$ being a surface, we define a certain space of handles to cut the surface along them. Finally if $M$ is a three manifold, we first reduce to the case of irreducible three manifolds and we cut it along incompressible surfaces in a ``contractible space of choices". For this reason, we consider the case of three manifolds separately. To cut along submanifolds, we need to consider ``nicely" embedded submanifolds. This is more essential in dimension higher than $2$. So let us recall the definition of locally flat embeddings.
 \begin{defn} The $k$ simplices of the simplicial set of {\it locally flat embeddings} $\text{Emb}^{\text{lf}}_{\bullet}(N,M)$ is given by the commutative diagram 
   \[
     \begin{tikzpicture}[node distance=2cm, auto]
  \node (A) {$\Delta^k\times N$};
  \node (B) [right of=A] {$$};
  \node (C) [right of=B] {$\Delta^k\times M$};  
  \node (D) [below of=B, node distance=1.3cm] {$\Delta^k,$};
  \draw[->] (C) to node {$pr_1$} (D);
  \draw [<-] (D) to node {$pr_1$} (A);
    \draw [->] (A) to node {$f$} (C);
\end{tikzpicture}
    \]
where $f$ is a homeomorphism onto its image which is also locally flat. To recall the condition of being locally flat,  let the codimension of the map $f$ be $p$. Then for all $(t, n) \in \Delta^k \times N$ there exist open neighborhoods $U$ and $V$ around $t$ and $n$ respectively such that there is a map $U\times V \times \bR^{p}\to \Delta^k\times N$ over $\Delta^k$ which extends $f|_{U\times V}$ and is a homeomorphism onto its image. If $N$ has a boundary, we consider those embeddings that restrict to locally flat embeddings of the interior and locally flat embeddings of the boundary.
 \end{defn} 
 \begin{rem} To cut  codimension $0$ submanifolds with boundary and obtain a manifold, we need to have a bicollared boundary. This is guaranteed by Brown's result \cite[Theorem 3]{brown1962locally} that locally flat two sided codimension 1 submanifolds are bicollared.  
 \end{rem}
\begin{rem}\label{emb}We can also consider the space of locally flat embeddings $\textit{Emb}^{\text{lf}}(N,M)$ as a subspace of embeddings of $N$ into $M$ with the compact-open topology. In codimension $0$ as in codimension $3$ and higher (see \cite[Appendix]{lashof1976} for the comparison between different versions of the embedding spaces), it is known that the realization of $\text{Emb}^{\text{lf}}_{\bullet}(N,M)$ has the same homotopy type as $\textit{Emb}^{\text{lf}}(N,M)$ and in fact in these cases $\text{Emb}^{\text{lf}}_{\bullet}(N,M)$ is equal to the singular set $S_{\bullet}(\textit{Emb}^{\text{lf}}(N,M))$.
\end{rem}
\section{Cutting surfaces into disks} \label{disk} In this section $M$ is an oriented surface with possibly nonempty boundary.
\subsection{$0$-handle resolutions}The first step is to reduce the statement of \Cref{main}  to the case of the surfaces with non-empty boundary so that one could  remove $1$-handles. Hence we first want to remove  disks ($0$-handles) from a closed surface $M$. To parametrize different choices of removing $0$-handles, we define the following semisimplicial spaces. 
\begin{defn} \label{def1}We first define the semisimplicial simplicial set on which the simplicial group $S_{\bullet}(\tH_{0,\partial}(M))$ acts.
\begin{itemize}[leftmargin=*]\item {\bf Topological versions:}
Let $[p]$ denote the set $\{0,1,...,p\}$ of $p+1$ ordered elements. 
\begin{itemize} 
\item Let $$ A_p(M)_{\bullet}=\text{\textnormal{Emb}}^{\text{lf}}_{\bullet}(\coprod_{[p]} D^2, M)$$ denote the simplicial set of locally flat embeddings  consisting of  orientation preserving  embeddings of $p$ disjoint closed unit $2$-disks into $M$.  The collection $A_{\bullet}(M)_{\bullet}$ is a semisimplicial simplicial set where the face maps in the semisimplicial direction are given by forgetting disks.  We shall write $A_p(M)$ for the realization of $ A_p(M)_{\bullet}$ in the simplicial direction. 

\item Let $A^{\bf t}_p(M)$  denote the space $\textit{Emb}^{\text{lf}}(\coprod_{[p]} D^2, M)$ equipped with the compact-open topology. By \Cref{emb}, the natural map $A_p(M)\to  A^{\bf t}_p(M)$ is a weak equivalence. Let $A^{{\bf t},\delta}_p(M)$ be the underlying set of the space $A^{\bf t}_p(M)$, in other words, we have $A^{{\bf t},\delta}_p(M)=A_p(M)_0$.

\item We also define an auxiliary semisimplicial simplicial set  $\overline{A}_{\bullet}(M)_k$ whose set of $0$-simplices in semisimplicial direction is the same as $A_0(M)_k$ but its $p$-simplices in semisimplicial direction have $k$-simplices  consisting of $(p+1)$-tuples $(\phi_0(t),\phi_1(t),\dots,\phi_p(t))$ of $k$-simplices of $A_0(M)_k$ where for all $t\in \Delta^k$ the centers of the embedded disks $\phi_i(t)$ are pairwise disjoint but the disks may overlap. We shall write $\overline{A}_p(M)$ for the realization of $ \overline{A}_p(M)_{\bullet}$ in the simplicial direction. 
\end{itemize}
% The simplicial group $S_{\bullet}(\tH_{0}(M,\partial M))$ acts on the simplical set $A_p(M)$. The $0$-handle resolution of $\mathrm{B} |S_{\bullet}(\tH_{0}(M,\partial M))|$ is defined to be the augmented semisimplicial space
%\[
%X_{\bullet}(M)= A_{\bullet}(M)\hcoker \tH_0(M)\to \BH_0(M).
%\]
\item {\bf Discrete version:} In the $0$-simplices  $ \text{\textnormal{Emb}}^{\text{lf}}_0(\coprod_{[p]} D^2, M)$  of the simplical set $A_p(M)_{\bullet}$, we say two embeddings $g_1$ and $g_2$  have the same germ if there exists an open neighborhood $U\subset D^2$ around the origin so that $g_1|_{\coprod_{[p]} U}= g_2|_{\coprod_{[p]} U}$. 
\begin{itemize}
\item Let
\[
A^{\delta}_{\bullet}(M)= \text{\textnormal{Emb}}_0^{\text{g}}(\coprod_{[\bullet]} D^2, M),
\]
denote the set of germs of embeddings of disjoint union of $p+1$ disks compatible with the orientation of $M$. 
\item We define an auxiliary semisimplicial set $\overline{A}^{\delta}_{\bullet}(M)$ which is given by $0$-simplices in the simplicial direction of the semisimplicial simplicial set $\overline{A}_{\bullet}(M)_{\bullet}$.
\end{itemize}
%Also the $0$-handle resolution for $\BdH_0(M)$ is the augmented semisimplicial space
%\[
%X_{\bullet}^{\delta}(M):= A_{\bullet}^{\delta}(M)\hcoker \tdH_0(M)\to \BdH_0(M),
%\]
\end{itemize}
\end{defn}
The simplicial group $S_{\bullet}(\tH_{0, \partial}(M))$ acts on the simplical set $A_p(M)_{\bullet}$. To define the $0$-handle resolution of $\mathrm{B} |S_{\bullet}(\tH_{0, \partial}(M))|$, we need to consider the homotopy quotient of the action of $S_{\bullet}(\tH_{0,\partial}(M))$ on $A_p(M)$. To do so, we recall the two sided bar construction. Recall that for a group $G$ acting on a topological space $X$, the two sided bar construction $B_{\bullet}(X,G,*)=X\times G^{\bullet}$ is a simplicial space with the usual face maps and degeneracies. For a discrete group (or well-pointed topological group), the realization of this bar construction is a model for the homotopy quotient $X\hcoker G$. 
\begin{defn}\label{defn1}
We define the $0$-handle resolution $X_{\bullet, n,k}(M)$ to be an augmented semisimplicial bisimplicial set
\[
X_{\bullet, n,k}(M)=B_{n}(A_{\bullet}(M)_k,S_{k}(\tH_{0,\partial}(M)),*)\xrightarrow{\epsilon} B_{n}(*,S_{k}(\tH_{0,\partial}(M)),*).
\] 
We denote the realization of $X_{p, n,k}(M)$ in the bisimplicial directions by $X_p(M)$. The face maps of $X_p(M)$ are induced by the face maps of $A_p(M)$ in the semisimplicial direction. 
\end{defn}
Since in realizing a bisimplicial set, the order of realization in each direction, does not matter (\cite[Lemma, Page 10]{MR0338129}), the realization of the augmentation map $\epsilon$ is given by
\[
A_p(M)\hcoker |S_{\bullet}(\tH_{0,\partial}(M))|\xrightarrow{|\epsilon|} \mathrm{B}|S_{\bullet}(\tH_{0,\partial}(M))|\xrightarrow{\simeq}\BH_{0,\partial}(M).
\]
 The semisimplicial space $X_p(M)$ is called a semisimplicial resolution for  the classifying space $\BH_0(M,\partial M)$, because, as we shall see in \Cref{claim}, if we realize $X_p(M)$ in the semisimplicial direction we obtain a map
\[
||\epsilon||\colon |X_{\bullet}(M)|\to \mathrm{B}|S_{\bullet}(\tH_{0,\partial}(M))|,
\]
which turns out to be a weak equivalence. The fiber of the map $||\epsilon||$ is the realization $|A_{\bullet}(M)|$ of the semisimplicial simplicial set $A_{\bullet}(M)_{\bullet}$ which means first realizing in the simplicial direction to obtain a semisimplicial space and then realizing in the semisimplicial direction. Using Proposition 2.18 in \cite{kupers2015proving}, to prove that the map $||\epsilon||$ is a weak equivalence, it is enough to show that its fiber $|A_{\bullet}(M)|$ is contractible, as we shall prove in \Cref{claim1}.
\begin{rem}
A more geometric  model for this homotopy quotient is to consider the simplicial set $$\text{Emb}^{\text{lf}}_{\partial}(M;\bR^{\infty})\times_{S_{\bullet}(\tH_{0,\partial}(M))} A_p(M)_{\bullet},$$ but we do not need this model for our argument.
\end{rem}
On the other hand, the discrete group $\tdH_{0,\partial}(M)$ acts on $A_{\bullet}^{\delta}(M)$. So we define the semisimplicial resolution for $\BdH_{0,\partial}(M)$ as follows.
\begin{defn}\label{defn2}
The $0$-handle resolution for $\BdH_{0,\partial}(M)$ is the augmented semisimplicial space
\[
\theta\colon X_{\bullet}^{\delta}(M):= A_{\bullet}^{\delta}(M)\hcoker \tdH_{0,\partial}(M)\to \BdH_{0,\partial}(M).
\]
\end{defn}
Similarly, to prove that $|\theta|\colon |X_{\bullet}^{\delta}(M)|\to \BdH_0(M)$ is a weak equivalence, it is enough to show that its fiber, the realization $|A_{\bullet}^{\delta}(M)|$, is contractible, as we shall prove it in \Cref{claim}.

Note that there are natural maps $$A_{\bullet}(M)_{\bullet}\to \overline{A}_{\bullet}(M)_{\bullet}\leftarrow \overline{A}^{\delta}_{\bullet}(M)\to A^{\delta}_{\bullet}(M) ,$$ where the first map is the inclusion. By scaling the disks, it is easy to see that the first map induces a weak equivalence after realization in the simplicial directions. The second map is the inclusion to the $0$-simplices in the simplicial direction and the last map is induced by taking germs of embdeddings of disks at their centers.  

\subsubsection{The homotopy type of $X_p(M)$ and $X_p^{\delta}(M)$} For the semisimplical space $X_p(M)$, we have a spectral sequence
\[
E^1_{p,q}(X_{\bullet}(M))=H_q(X_p(M))\Rightarrow H_{p+q}(|X_{\bullet}(M)|),
\]
for any coefficient systems that pulls back from $|X_{\bullet}(M)|$ (see \cite[Section 1.4]{ebert2017semi}) so we often suppress the coefficients for brevity. Similarly, we have a spectral sequence that calculates $H_{p+q}(|X_{\bullet}^{\delta}(M)|)$. In order to be able to compare these spectral sequences, we need to compare the homotopy types of $X_p(M)$ and $X_p^{\delta}(M)$. The first step is to make sure that they have the same number of connected components. 

The set of the connected components for $X_p^{\delta}(M)$ which is the homotopy quotient $A_{p}^{\delta}(M)\hcoker \tdH_{0,\partial}(M)$, is in bijection with the set of the orbits of the action of $\tdH_{0,\partial}(M)$ on $A_p^{\delta}(M)$. On the other hand, since the map $A_p(M)\xrightarrow{\simeq} A^{\bf t}_p(M)$ is equivariant with respect to the homomorphism $|S_{\bullet}(\tH_{0,\partial}(M))|\xrightarrow{\simeq} \tH_{0,\partial}(M)$, we obtain a map 
\[
X_p(M)=A_p(M)\hcoker |S_{\bullet}(\tH_{0,\partial}(M))| \to A_p^{\bf t}(M)\hcoker  \tH_{0,\partial}(M),
\]
which is a weak equivalence by the comparison of the long exact sequences of homotopy groups for fibrations. Therefore, the set of the connected components for $X_p(M)$ is in bijection with the set of the orbits of the action of $\tH_{0,\partial}(M)$ on $A_p^{\bf t}(M)$.
\begin{lem}\label{disks}The set of the connected components of $X_p(M)$ is in bijection with that of $X_p^{\delta}(M)$ for all $p$.
\end{lem}
\begin{proof}From the above discussion, it is enough to show that the set of the orbits of the action of $\tH_{0,\partial}(M)$ on $A_p^{\bf t}(M)$ is in bijection with that of the action of $ \tdH_{0,\partial}(M)$ on $A_{p}^{\delta}(M)$. As we shall see, these actions are transitive, but what matters which will be useful later when the action is not transitive, is that the set of the orbits is determined by the action on the core of the handles. In other words, in this case, the orbits are determined by the action on the center of the disks (or rather germs of the disks at their centers). 

Suppose we have two embeddings $e_1$ and $e_2$ in $A_p(M)_0$. Each embedding gives a configuration of $p+1$ disjoint unparameterized disks  $e_i(M)\subset M$. Since we work with orientation preserving embeddings and homeomorphisms, if we show that we can find an element $f$ in $\tH_{0,\partial}(M)$ that sends the unparameterized $e_1(M)$ to $e_2(M)$, we can change $f$ up to isotopy to send $e_1$ to $e_2$. We first arrange $f$ to send the centers of the disks in $e_1(M)$ to the centers of the disk in $e_2(M)$. This is easy by the fact that the action of $\tH_{0,\partial}(M)$ is strongly $k$-transitive for all $k$ (see \cite[Lemma 2.1.10]{MR1445290}) which means that the action of $\tH_{0,\partial}(M)$ on set of $k$-tuples of points in $M$ is transitive. Then by scaling the disks, we shall change $f$ up to isotopy to send $e_1(M)$ into $e_2(M)$. Since the embeddings are locally flat the regions between the disks are homeomorphic to an annulus (\cite{MR0242165}, of course in low dimensions, we do not need the full force of the annulus theorem), we can change $f$ up to an isotopy to send $e_1(M)$ to $e_2(M)$. Similarly for the action of $\tdH_{0,\partial}(M)$ on $A_p^{\delta}(M)$, if we have two germs of embeddings $e_1^g$ and $e_2^g$ of disks, we first choose representatives of germs and proceed as before. Therefore, the set of the orbits of both actions depend on the centers of the disks (cores of the handles) and since the actions on the centers are transitive, there is a bijection between the set of the orbits.  
\end{proof}

To find the homotopy type of $X_p(M)$ and $X_p^{\delta}(M)$, let us first recall a version of Shapiro's lemma. Let $G$ be a discrete group acting on a set $X$. One could decompose $X$ as $$\coprod_{\alpha\in \text{orbits}} G/H_{\alpha},$$ where $H_{\alpha}<G$ is a stabilizer subgroup of an element in the orbit $\alpha$. Then, we have a map
\begin{equation}\label{shapiro}
\coprod_{\alpha\in \text{orbits}} \mathrm{B}H_{\alpha}\xrightarrow{\simeq} X\hcoker G = |\mathrm{B}_{\bullet}(X,G,*)|,
\end{equation}
which is a homotopy equivalence. 

 Now let $e_p\in A^{\bf t}_p(M)$ be an embedding of $p+1$ disjoint disks and let $[e_p]\in A_p^{\delta}(M)$ denote its germ at its center. Let $\text{Stab}(e_p)$ denote the stabilizer group of $e_p$ under  the action of $\tH_{0,\partial}(M)$ on $A_p(M)$. We denote the stabilizer of $[e_p]$ under the action of $\tdH_{0,\partial}(M)$ on $A_p^{\delta}(M)$ by $\text{Stab}^{\delta}([e_p])$. Let also $\text{Stab}([e_p])$ denote the same group but with the subspace topology as a subgroup of $\tH_{0,\partial}(M)$. 
 \begin{lem}\label{lem1} There is a map $ \mathrm{B}\text{\textnormal{Stab}}^{\delta}([e_p])\to X^{\delta}_p(M)$ which is a homotopy equivalence.
 \end{lem}
 \begin{proof}This is implied by Shapiro's lemma and the fact the action is transitive in this case.
 \end{proof}
 \begin{lem}\label{lem2}There is a zig-zag of weak equivalences between $ \mathrm{B}\text{\textnormal{Stab}}(e_p)$ and  $X_p(M)$.
 \end{lem}
 \begin{proof}
 By the parametrized isotopy extension theorem in the topological setting (\cite[page 19]{burghelea1974homotopy}), we know that the map
 \[
S_{\bullet}(\tH_{0,\partial}(M))\xrightarrow{\text{ev}} A_p(M)_{\bullet}, 
 \]
that is induced by the action on a fixed element, is a Kan fibration\footnote{Note that in this case since $A_p(M)_{\bullet}=S_{\bullet}(A_p^{\bf t}(M))$, the fact that the map $\text{ev}$ is a Kan fibration implies that the map $\tH_{0,\partial}(M)\to A_p^{\bf t}(M)$ is a Serre fibration}. The fiber of this map is $S_{\bullet}(\text{\textnormal{Stab}}(e_p))$. Thus, for each $k$, we have a bijection between the set $A_p(M)_k$ and the coset  $$S_k(\tH_{0,\partial}(M))/S_{k}(\text{\textnormal{Stab}}(e_p)).$$ So again by Shapiro's lemma, we have a simplicial map
\[
\mathrm{B}S_{k}(\text{\textnormal{Stab}}(e_p))\to A_p(M)_k\hcoker S_k(\tH_{0,\partial}(M))=|\mathrm{B}_{\bullet}(A_p(M)_k, S_k(\tH_{0,\partial}(M)), *)|,
\]
which is a weak equivalence for all $k$. So again by using  \cite[Lemma, Page 10]{MR0338129}, that the order of realizations for a bisimplicial set does not matter, if we realize in the $k$-direction, we obtain 
\[
 \mathrm{B}\text{\textnormal{Stab}}(e_p)\xleftarrow{\simeq}\mathrm{B}|S_{\bullet}(\text{\textnormal{Stab}}(e_p))|\to X_p(M)= A_p(M)\hcoker |S_{\bullet}(\tH_{0,\partial}(M))|,
\]
which is a weak equivalence. 
 \end{proof}
 Now we need a lemma from homotopy theory to show that the weak homotopy type of $|X_{\bullet}(M)|$ and $|X_{\bullet}^{\delta}(M)|$ are the same as  $\BH_{0,\partial}(M)$ and $\BdH_{0,\partial}(M)$ respectively. 
 
 \subsubsection{A lemma in homotopy theory}
Here the goal is to show that $ |\overline{A}^{\delta}_{\bullet}(M)|$ and $|A_{\bullet}(M)|$ are weakly contractible. Proving that the realization of the discrete version $ |\overline{A}^{\delta}_{\bullet}(M)|$ is contractible is easier. Using a  lemma  in homotopy theory, we show that the contractibility of $ |\overline{A}^{\delta}_{\bullet}(M)|$ implies the weak contractibility of  $|A_{\bullet}(M)|$. This technique is originally due to Segal (\cite[Appendix]{segal1978classifying}) and it is reformulated by Weiss in  \cite[Lemma 2.2]{weiss2005does}. In particular, in the setting of semi-simplicial spaces, we use an application of this technique (\cite[Proposition 2.8]{galatius2014homological}) due to Galatius and Randal-Williams.
\begin{prop}\label{claim}
The realizations $|A^{\delta}_{\bullet}(M)|$ and $ |\overline{A}^{\delta}_{\bullet}(M)|$ are weakly contractible.
\end{prop}
\begin{proof}
We give a proof for weak contractibility of $|A^{\delta}_{\bullet}(M)|$, the case of $ |\overline{A}^{\delta}_{\bullet}(M)|$ is similar. Let $ f: S^k\to |A^{\delta}_{\bullet}(M)|$ be an element in the $k$-th homotopy group of $|A^{\delta}_{\bullet}(M)|$. Since $|A^{\delta}_{\bullet}(M)|$ is a CW-complex and $S^k$ is compact, the map $f$ hits finitely many simplices of $|A^{\delta}_{\bullet}(M)|$. Hence, there exists a point   ${\bf p}$ and an embedded disk $e(D^2)$ around it such that as an element of $A^{\delta}_{0}(M)$ is not hit by the map $f$. Thus, we have $f(S^k)\subset |A^{\delta}_{\bullet}(M\backslash e(D^2))|$. Adding the germ of $e$ at ${\bf p}$ to the list of germs of embeddings of disks in $M\backslash e(D^2)$ gives a semisimplicial null-homotopy for the inclusion $A^{\delta}_{\bullet}(M\backslash e(D^2))\hookrightarrow A^{\delta}_{\bullet}(M)$. Therefore, the element $f(S^k)$ can be coned off inside $|A^{\delta}_{\bullet}(M)|$.
\end{proof}
\begin{rem}
Note that because $|A^{\delta}_{\bullet}(M)|$ and $ |\overline{A}^{\delta}_{\bullet}(M)|$ have CW structures, so they are in fact contractible.
\end{rem}
Using Proposition 2.18 in \cite{kupers2015proving}, to prove that the maps 
\begin{align}\label{eee}
%\begin{gathered}
|\theta|\colon & |X_{\bullet}^{\delta}(M)|\xrightarrow{} \BdH_0(M),\\
||\epsilon||\colon & |X_{\bullet}(M)|\xrightarrow{}\mathrm{B}|S_{\bullet}(\tH_{0,\partial}(M))|.
%\end{gathered}
\end{align}
 are  weak equivalences, it is enough to show that their fibers $|A^{\delta}_{\bullet}(M)|$ and $|A_{\bullet}(M)|$ respectively are contractible.

 Therefore, by \Cref{claim} the first map $|X_{\bullet}^{\delta}(M)|\xrightarrow{\simeq} \BdH_0(M)$ is a weak homotopy equivalence. To prove that the second map is also a weak homotopy equivalence, we need to show that $|A_{\bullet}(M)|$ is weakly contractible. To do so, we use the bisimplicial technique due to Quillen \cite[Proof of Theorem A]{MR0338129}. First note that since the map
\[
A_{\bullet}(M)\xrightarrow{\simeq} \overline{A}_{\bullet}(M)
\]
is a levelwise weak equivalence, it induces a weak homotopy equivalence between the realizations in semisimplicial directions. Hence, to show that $| A_{\bullet}(M)|$ is weakly contractible, it is enough to show that in the zig-zag
\begin{equation}\label{eq:3}
A_{\bullet}(M)\xrightarrow{\simeq} \overline{A}_{\bullet}(M)\xleftarrow{\beta} \overline{A}^{\delta}_{\bullet}(M)
\end{equation}
the second map $\beta$ induces a weak homotopy equivalence after realizations in semisimplicial directions.  

%Note that $\beta$ is equivariant with respect to the map $\tdH_0(M)\to \tH_0(M)$ and the first map is equivariant with respect to the action of $\tH_0(M)$ on its both sides.
\begin{defn}\label{biss}
Let $A_{\bullet,\bullet}(M)_k$ be the bisemisimplicial simplicial set such that $A_{p,q}(M)_k$ is the subset of  $\overline{A}^{\delta}_{p}(M)\times  \overline{A}_{q}(M)_k$ consisting of those $(p+q+2)$-tuples $$(a_0,\dots,a_p,c_0,\dots,c_q),$$ where the centers of the disks $a_i$ and the disks $c_j(t)$ are pairwise disjoint for all $t\in \Delta^k$.
\end{defn}
The bisemisimplicial simplicial set $A_{\bullet,\bullet}(M)_k$ is augmented in two different semisimplicial directions
\[
\epsilon_{p,k}\colon A_{p,\bullet}(M)_k\to \overline{A}^{\delta}_{p}(M),
\]
\[
\delta_{q,k}\colon  A_{\bullet,q}(M)_k\to \overline{A}_{q}(M)_k.
\]
Let $A_{\bullet,\bullet}(M)$ be the bisemisimplicial space obtained by realizing $A_{\bullet,\bullet}(M)_k$ in the simplicial direction. Similar to \cite[Lemma 5.8]{galatius2014homological}, one can show  that the following diagram is homotopy commutative
 \begin{equation}\label{commdiagram}\begin{gathered}
     \begin{tikzpicture}[node distance=2cm, auto]
  \node (A) {$|\overline{A}^{\delta}_{\bullet}(M)|$};
  \node (B) [right of=A] {$$};
  \node (C) [right of=B] {$|\overline{A}_{\bullet}(M)|$};  
  \node (D) [below of=B, node distance=1.3cm] {$|A_{\bullet,\bullet}(M)|.$};
  \draw[<-] (C) to node {$\delta$} (D);
  \draw [->] (D) to node {$\epsilon$} (A);
    \draw [->] (A) to node {$$} (C);
\end{tikzpicture}
\end{gathered}
\end{equation}

\begin{prop}\label{claim1}
The realization $|A_{\bullet}(M)|$ is weakly contractible.
\end{prop}
\begin{proof}
Since $|A_{\bullet}(M)|\xrightarrow{\simeq}|\overline{A}_{\bullet}(M)|$, we instead show that $|\overline{A}_{\bullet}(M)|$ is weakly contractible. Because the diagram (\ref{commdiagram}) is homotopy commutative and $|\overline{A}^{\delta}_{\bullet}(M)|$ is weakly contractible, if we show that the map $\delta$ is a weak homotopy equivalence, we then deduce that $|\overline{A}_{\bullet}(M)|$ is also weakly contractible. To do so, it suffices to prove that 
\[
|\delta_q|\colon |A_{\bullet,q}(M)|\to \overline{A}_{q}(M),
\] 
is a weak equivalence. The idea is to show that $|\delta_q|$ is a microfibration with a contractible fiber. But since $\overline{A}_{q}(M)$ is the realization of a simplicial set, we shall apply simplicial approximation similar to \cite[Proposition 2.44]{kupers2015proving} to exhibit the map $|\delta_q|$ as the realization of a map between simplicial sets.  Note that $A_{\bullet,q}(M)_k$ is a semisimplicial set for a fixed $q$ and $k$. We can freely add all degeneracies (see \cite[Section 1]{ebert2017semi}) to obtain a bisimplicial set $EA_{\bullet,q}(M)_{\bullet}$ whose realization is homotopy equivalent to $|A_{\bullet,q}(M)|$. More concretely, we have 
\[
EA_{p,q}(M)_{k}= \coprod_{[p]\twoheadrightarrow [p']}A_{p',q}(M)_{k}.
\]
The realization of the bisimplicial set $EA_{\bullet,q}(M)_{\bullet}$ is homeomorphic to the realization of its diagonal $\text{diag}(EA_{\bullet,q}(M)_{\bullet})$. Therefore, it is enough to show that the augmentation map 
\[
\delta_{q,\bullet}\colon \text{diag}(EA_{\bullet,q}(M)_{\bullet})\to \overline{A}_{q}(M)_{\bullet},
\]
induces a weak equivalence  after realization. By the simplicial approximation, it is enough to show that for each pair $(K,\partial K)$ of simplicial sets where $(|K|, |\partial K|)\cong (D^i, S^{i-1})$ and each diagram
 \[
 \begin{tikzpicture}[node distance=2.6cm, auto]
  \node (B)  {$\partial K$};
  \node (C) [below of= B, node distance=1.6cm ] {$K$};  
  \node (E) [right of= B] {$\text{diag}(EA_{\bullet,q}(M)_{\bullet})$};
  \node (F) [right of= C ] {$\overline{A}_{q}(M)_{\bullet}$};  
  \draw [->] (B) to node {$g$}(E);
  \draw [->] (C) to node {$G$}(F);
  \draw [->] (B) to node {$$} (C);
  \draw [->] (E) to node {$\delta_{q,\bullet}$} (F);

\end{tikzpicture}
\]
we have a lift $\tilde{G}\colon |K|\to |\text{diag}(EA_{\bullet,q}(M)_{\bullet})|$ so that $\tilde{G}|_{|\partial K|}=|g|$. The map $G$ can be represented by a locally flat immersion $f\colon D^i\times \coprod_{[q]} D^2\to D^i\times M$ over $D^i$ so that the centers of the embedded disks are disjoint. 

On the other hand, by \Cref{biss}, the map $|g|$ gives a map $h\colon S^{i-1}\to |\overline{A}^{\delta}_{\bullet}(M)|$ where for each $t \in S^{i-1}$ the center of $g(t)$ and the centers ${\bf c}_t$ of the embedded disks $f|_t$ are disjoint. We want to show that $h$ can be extended to a map $\tilde{h}\colon D^i\to |\overline{A}^{\delta}_{\bullet}(M)|$ where for each $t \in D^{i}$ the center of $g(t)$ and ${\bf c}_t$ are disjoint. Note that $h$ gives an element of the homotopy group of the space of pairs $$X=\{(t,x)\in D^i\times |\overline{A}^{\delta}_{\bullet}(M)| | \text{ the center of $x$ and ${\bf c}_t$ are disjoint} \}.$$
Hence, it suffices to show that $X$ is contractible. Note that the projection $X\to D^i$ is a microfibration by the openness of the condition of centers being disjoint. And the fiber over $t$ is homeomorphic to $|\overline{A}^{\delta}_{\bullet}(M\backslash {\bf c}_t)|$ which is contractible by \Cref{claim}. Therefore, the projection is a fibration (see \cite[Lemma 2.2]{weiss2005does} or \cite[Proposition 2.6]{galatius2014homological}) with a contractible fiber so $X$ is contractible. Hence, since $|\delta_q|$ is a weak equivalence for all $q$, so is $\delta$.
\end{proof}
\subsubsection{Reducing \Cref{main} to the case of manifolds with boundary}\label{redbdry} Recall the goal is to compare the spectral sequences for the semisimplicial spaces $X_{\bullet}(M)$ and $X_{\bullet}^{\delta}(M)$. For these $0$-handle resolutions (unlike the $1$-handle resolutions as we shall see later), there is no direct semisimplicial map from $X_{\bullet}^{\delta}(M)$ to $X_{\bullet}(M)$. But we shall find a zig-zag of semisimplicial maps between them and show that our zig-zag of map induces a map between their spectral sequences.

Since $A^{{\bf t},\delta}_{\bullet}(M)=A_{\bullet}(M)_0$, the inclusion  $A^{{\bf t},\delta}_{\bullet}(M)\to A_{\bullet}(M)_{\bullet}$ is equivariant with respect to the map $\tdH_{0,\partial}(M)\to S_{\bullet}(\tH_{0,\partial}(M))$. Therefore, we have an induced map between homotopy quotients
\begin{equation}\label{eq:4}
\alpha_{\bullet}\colon A^{{\bf t},\delta}_{\bullet}(M)\hcoker \tdH_0(M) \to X_{\bullet}(M). 
\end{equation}
On the other hand, the map $A^{{\bf t},\delta}_{\bullet}(M)\to A^{\delta}_{\bullet}(M)$ is equivariant with respect to the action $\tdH_{0,\partial}(M)$. Therefore, we have an induced map between homotopy quotients
\begin{equation}\label{eq:4'}
 A^{{\bf t},\delta}_{\bullet}(M)\hcoker \tdH_{0,\partial}(M)\to X_{\bullet}^{\delta}(M).
\end{equation}
So we have the following homotopy commutative diagram
{\small
\begin{equation}\label{e'}
\begin{gathered}
\begin{tikzcd}
X_p^{\delta}(M)& A^{{\bf t},\delta}_p(M)\hcoker \tdH_{0,\partial}(M)\arrow{l}\arrow[""]{r}&X_{p}(M) \\ \mathrm{B}\text{\textnormal{Stab}}^{\delta}([e_p])\arrow["\simeq"]{u}&\mathrm{B}\text{\textnormal{Stab}}^{\delta}(e_p)\arrow[""]{l}\arrow["\simeq"]{u}\arrow[""]{r}& \mathrm{B}|S_{\bullet}(\text{\textnormal{Stab}}(e_p))|\arrow["\simeq"]{u},
\end{tikzcd}
 \end{gathered}
 \end{equation}}
\noindent where the first and the last weak equivalences are given by \Cref{lem1} and \Cref{lem2} and the middle weak equivalence is deduced from the transitivity of the action of $\tdH_{0,\partial}(M)$ on $A^{{\bf t},\delta}_{\bullet}(M)$ and the Shapiro's lemma. 
 \begin{lem}\label{lem3}
 The map $\mathrm{B}\text{\textnormal{Stab}}^{\delta}(e_p)\to \mathrm{B}\text{\textnormal{Stab}}^{\delta}([e_p])$ induces a homology isomorphism.
 \end{lem}
Note that \Cref{lem3} implies that the spectral sequences of the semisimplicial spaces $X_{\bullet}^{\delta}(M)$ and $A^{{\bf t},\delta}_{\bullet}(M)\hcoker \tdH_{0,\partial}(M)$ are isomorphic as the map (\ref{eq:4'}) induces an isomorphism on $E^1$-page. On the other hand $\alpha_{\bullet}$ induces a map between spectral sequences for $A^{{\bf t},\delta}_{\bullet}(M)\hcoker \tdH_{0,\partial}(M)$ and $X_{\bullet}(M)$. With abuse of notation, let $\alpha_*$ denote the induced map between spectral sequences for $X_{\bullet}^{\delta}(M)$ and $X_{\bullet}(M)$. 
\begin{equation}\label{eq:5}
\begin{gathered}
\begin{tikzcd}
H_q(X_{p}^{\delta}(M))\arrow["\alpha_*"]{r}\arrow[Rightarrow]{d}&H_q(X_{p}(M))\arrow[Rightarrow]{d}\\H_{p+q}(|X_{\bullet}^{\delta}(M)|)\arrow[""]{r}\arrow["\cong"]{d}& H_{p+q}(|X_{\bullet}(M)|)\arrow["\cong"]{d}\\ H_{p+q}(\BdH_{0,\partial}(M))\arrow[""]{r}& H_{p+q}(\mathrm{B}|S_{\bullet}(\tH_{0,\partial}(M))|).
\end{tikzcd}
\end{gathered}
\end{equation}
Before proving \Cref{lem3}, let us show that this comparison of spectral sequences reduces \Cref{main} from the case of closed surfaces to  surfaces with non-empty boundary. 
\begin{defn}\label{punctures}
Let $e_p$ be an element in $A_p(M)$. Let $M\backslash e_p$ denote the manifold obtained from $M$ by removing the interior of the embedded disks in $M$ given by $e_p$. Also let $M\backslash c(e_p)$ denote the punctured manifold obtained from $M$ by removing the centers of the embedded disks given by $e_p$. 
\end{defn}
\begin{prop}\label{reduction1}Suppose \Cref{main} holds for $M\backslash e_p$ for all $e_p\in A_p(M)$ and all $p$. Then it also holds for $M$. 
\end{prop}
\begin{proof}
Given the spectral sequence (\ref{eq:5}), it suffices to prove that $\alpha_*$ induces an isomorphism between $E^1$-pages. Using the commutative diagram (\ref{e'}) and \Cref{lem3}, it is enough to show that the hypothesis of the proposition implies that the map
\[
\mathrm{B}\text{\textnormal{Stab}}^{\delta}(e_p)\to \mathrm{B}\text{\textnormal{Stab}}(e_p),
\]
induces a homology isomorphism. Note that the identity component of $\text{\textnormal{Stab}}(e_p)$ is $\tH_{0,\partial}(M\backslash e_p)$, so we have a short exact sequence of groups
\[
1\to \tH_{0,\partial}(M\backslash e_p)\to \text{\textnormal{Stab}}(e_p)\to \pi_0(\text{\textnormal{Stab}}(e_p))\to 1.
\]
From this short exact sequence, we obtain a homotopy commutative diagram between two fibrations
 \begin{equation}\label{ee}
 \begin{tikzcd}
 \BdH_{0,\partial}(M\backslash e_p)\arrow[" "]{r}\arrow[rightarrow]{d}& \BH_{0,\partial}(M\backslash e_p)\arrow[rightarrow]{d}\\\mathrm{B}\text{\textnormal{Stab}}^{\delta}(e_p)\arrow[""]{r}\arrow{d}& \mathrm{B}\text{\textnormal{Stab}}(e_p)\arrow{d}\\ \mathrm{B}\pi_0(\text{\textnormal{Stab}}(e_p))\arrow["\cong"]{r}& \mathrm{B}\pi_0(\text{\textnormal{Stab}}(e_p)).
\end{tikzcd}
 \end{equation}
 
 Now by the hypothesis, the map between fibers induces a homology isomorphism. Since the bases are the same, using the Serre spectral sequence, we conclude that the map between total spaces induces a homology isomorphism.
\end{proof}
\begin{proof}[Proof of \Cref{lem3}] Let us first consider $\text{\textnormal{Stab}}(e_p)$ and $\text{\textnormal{Stab}}([e_p])$ as  subgroups of $\tH_{0,\partial}(M)$ with the subspace topology. The identity components are respectively $\tH_{0,\partial}(M\backslash e_p)$ and $\tH_{0,c}(M\backslash c(e_p))$ where the latter is the identity component of $\tH_{c}(M\backslash c(e_p))$, the compactly supported homeomorphisms of the punctured surface $M\backslash c(e_p)$. Hence, we have a map between two fibrations
 \[
 \begin{tikzcd}
 \BdH_{0,\partial}(M\backslash e_p)\arrow[" "]{r}\arrow[rightarrow]{d}& \BdH_{0,c}(M\backslash c(e_p))\arrow[rightarrow]{d}\\\mathrm{B}\text{\textnormal{Stab}}^{\delta}(e_p)\arrow[""]{r}\arrow{d}& \mathrm{B}\text{\textnormal{Stab}}^{\delta}([e_p])\arrow{d}\\ \mathrm{B}\pi_0(\text{\textnormal{Stab}}(e_p))\arrow[""]{r}& \mathrm{B}\pi_0(\text{\textnormal{Stab}}([e_p])).
\end{tikzcd}
 \]
 The pushing collar lemma in \cite[Corollary 2.3]{nariman2014homologicalstability}  implies that the map between fibers induces a homology isomorphism. So if we show that the map between bases induces a weak equivalence, the lemma follows from a Serre spectral sequence argument again. 
 
 Let us first show that the map $\pi_0(\text{\textnormal{Stab}}(e_p))\to \pi_0(\text{\textnormal{Stab}}([e_p]))$ is surjective. Let $f\in \text{\textnormal{Stab}}([e_p])<\tH_{c}(M\backslash c(e_p))$, we shall change $f$ up to isotopy so that it fixes the disks $e_p(\coprod_{[p]}D^2)$. Since $f$ is supported away from the punctures $c(e_p)$, there exists a small disk $D^2_{\epsilon}\subset D^2$ of radius $\epsilon$ so that $f$ is supported away from $e_p(\coprod_{[p]}D^2_{\epsilon})$. Let $M_{\epsilon}$ denote the manifold obtained by removng the interior of $e_p(\coprod_{[p]}D^2_{\epsilon})$ from $M$. Then there is a self-embedding $p$ of $M_{\epsilon}$ isotopic to the identity such that $p(M_{\epsilon})=M\backslash e_p$. Consider the homeomorphism 
 \[
h_{\epsilon}(f)(x):= \begin{cases}
p\circ f\circ p^{-1}(x) & \text{if } x\in M\backslash e_p\\ \text{id} &\text{if }x\in e_p(\coprod_{[p]}D^2).
\end{cases}
\]
Since  $h_{\epsilon}(f)(x)$ is isotopic to $f$ and is the identity on $e_p(\coprod_{[p]}D^2)$, we have $h_{\epsilon}(f)(x)\in \text{\textnormal{Stab}}(e_p)$.

Proving that $\pi_0(\text{\textnormal{Stab}}(e_p))\to \pi_0(\text{\textnormal{Stab}}([e_p]))$ is injective is also the same. If we have $f_0$ and $f_1$ in $\text{\textnormal{Stab}}(e_p)$ that are isotopic in $\text{\textnormal{Stab}}([e_p])$, then the isotopy $f_t$ is supported away from $e_p(\coprod_{[p]}D^2_{\epsilon})$ for some small $\epsilon$. Then similar argument as the surjectivity case would imply that $f_0$ and $f_1$ are isotopic with an isotopy which is the identity on $e_p(\coprod_{[p]}D^2)$. 
\end{proof}
\begin{rem}\label{Jekel}
Note that in dimension $1$, by Mather's theorem (\cite{MR0288777}), we know that \Cref{main} holds for the intervals. Using the $0$-handle resolution for $S^1$ and \Cref{reduction1}, we deduce the Thurston theorem \ref{main} for $M=S^1$ (see \cite[Theorem 4]{MR2871163} for a similar idea). 
\end{rem}
\subsection{$1$-handle resolutions}\label{high} Using \Cref{reduction1}, to prove \Cref{main}, we can assume that $M$ is a surface with non-empty boundary. Now, we want to inductively reduce to the case of a simpler surface by removing $1$-handles from $M$. Similar to the previous section, to do this reduction, we need to define augmented semisimplicial sets whose realizations are contractible. 
\begin{defn}\label{def2} Let $\phi\colon D^1\times \bR\hookrightarrow M$ be a fixed $1$-handle so that $\phi(D^1\times \bR)\cap  \partial M=\phi(S^0\times \bR)$ and the core of the handle is the arc $\phi(D^1\times \{0\})$ in $M$. We shall write $\vec{e}$ for a unit basis vector in $\bR$.
\begin{figure}[ht]

\begin{tikzpicture}[scale=.5]

%\begin{scope}[shift={(6,0)}]
%\draw[line width=1.05pt] [dashed] (0,.-2.5) arc (-90:90:0.5 and 2.5);
%\draw [line width=1.05pt] (0,2.5) arc (90:270:0.5 and 2.5);
%%\draw [line width=1.05pt] (0,0)+(0,-2.5) arc (-90:90:10.5 and 2.5);
%\draw [line width=1.1pt] (4.6,.5) arc (240:300:2.5 and 6.75);
%\draw   [line width=1.7pt] (6.87, 0.13) arc (48:150:1.2 and 0.9);
%
%\draw  [red, ultra thick] (-0.5,0) to [out= 20, in=140] (3.5,0);
%
%\draw [ thick] (-3.5,-0.5)--(-3.5,0.5);
%
%\end{scope}
 \draw[line width=1.05pt] [dashed] (0,.-2.5) arc (-90:90:0.5 and 2.5);
\draw [line width=1.05pt] (0,2.5) arc (90:270:0.5 and 2.5);
%\draw [line width=1.05pt] (0,0)+(0,-2.5) arc (-90:90:14.5 and 2.5);
\draw [line width=1.05pt] (0,2.5)--(13,2.5);
\draw [line width=1.05pt] (13,-2.5) arc (-90:90:2.5);
\draw [line width=1.05pt] (0,-2.5)--(13,-2.5);
\draw [line width=1.1pt] (5.1,.0) arc (250:295:2.5 and 6.75);
\draw   [line width=1.7pt] (6.87, 0.13) arc (48:150:1.2 and 0.9);

\begin{scope}[shift={(6,0)}]
\draw [line width=1.05pt] (5.1,.0) arc (250:295:2.5 and 6.75);
\draw   [line width=1.7pt] (6.87, 0.13) arc (48:145:1.2 and 0.9);
\end{scope}

\draw  [line width=2mm, gray] (-0.5,0) to [out= -20, in=-140] (5.1,.0);
\draw  [line width=2mm, gray] (0.5,1) to [out= 20, in=140] (5.1,.0);

\draw  [] (-0.5,0) to [out= -20, in=-140] (5.1,.0);
\draw  [dashed] (0.5,1) to [out= 20, in=140] (5.1,.0);
\end{tikzpicture}  
  \caption{The $1$-handle $\phi(D^1\times \bR)$. }
    \label{core}

 \end{figure}

\noindent{\bf Topological version:} We define a semisimplicial simplicial set $B_{\bullet}(M, \phi)_{\bullet}$  and a semisimplicial space $B^{\bf t}_{\bullet}(M,\phi)$ on which $S_{\bullet}(\tH_{0,\partial}(M))$ and $\tH_{0,\partial}(M)$ act respectively.
\begin{itemize}
\item  We first define the $0$-simplices in the semisimplicial direction, $B_0(M,\phi)_{\bullet}$, to be the simplicial set given by pairs $(t, f)$ where $f\in \text{Emb}_{\bullet}^{\text{lf}}(D^1, M)$, $t\in \bR$ and for all $s\in \Delta^{\bullet}$, the embedded arc $f(s)$ satisfies
\begin{equation}\label{cond1}
f(s)|_{S^0}=\phi|_{S^0\times \{t \,\vec{e}\}},
\end{equation}
and the embedded arc $f(s)$ is isotopic to the arc $\phi(D^1\times \{t\,\vec{e}\})$ relative to the boundary.

The set of $p$-simplices, $B_p(M,\phi)_{\bullet}$, in the semisimplicial direction is given by $(p+1)$-tuples  $((t_0,f_0), (t_1, f_1),\dots, (t_p, f_p))$ so that $t_0<t_1<\cdots<t_p$ and for each $s\in \Delta^{\bullet}$, the arcs $f_i(s)$ are disjoint for all $i$. The face maps in the semisimplicial direction are given by forgetting the pairs $(t_i,f_i)$. We shall write $ B_{\bullet}(M,\phi)$ for the semisimplicial space obtained by realizing in the simplicial direction.
\item Let $B^{\bf t}_0(M,\phi)$ be the space of pairs $(t,f)$ where $t$ is a real number and $f\in \textit{Emb}^{\text{lf}}(D^1,M)$ such that $f$ satisfies the equation (\ref{cond1}). The space of such pairs are topologized as a subspace of $\bR\times  \textit{Emb}^{\text{lf}}(D^1,M)$. The space of $p$-simplices $B^{\bf t}_0(M,\phi)$ is given by $(p+1)$-tuples  $((t_0,f_0), (t_1, f_1),\dots, (t_p, f_p))$ so that $t_0<t_1<\cdots<t_p$ and the arcs $f_i$'s are disjoint for all $i$. The face maps for the semisimplicial space $B^{\bf t}_{\bullet}(M,\phi)$ are similarly given by forgetting the pairs $(t_i,f_i)$.
\end{itemize}

\noindent{\bf Discrete version:} Let $B^{\delta}_{\bullet}(M,\phi)$ be the semisimplicial set $B_{\bullet}(M,\phi)_{0}$ on which the discrete group $\tdH_{0,\partial}(M)$ acts. 
\end{defn} 
\begin{rem}
Since we want to emphasize on the methods for the possible applications in higher dimensions, we work with the locally flat embeddings, but in fact in dimension $2$ all embedded arcs are locally flat by the Schoenflies theorem.
\end{rem}
Note that by definition, for every pair $(t,f)\in B_0(M,\phi)_{\bullet}$, the real number $t$ is uniquely determined by $f$. We denote this $t$-coordinate by $t_{f}$. Moreover each simplex in $|B^{\delta}_{\bullet}(M,\phi)|$ has a canonical ordering induced by the condition $t_0<t_1<\cdots<t_p$. Therefore, $|B^{\delta}_{\bullet}(M,\phi)|$ has a simplicial complex structure with a natural ordering on each simplex. 

\begin{defn}\label{defn3}Similar to \Cref{defn1}, we define $Y_{\bullet, n,k}(M,\phi)$ to be an augmented semisimplicial bisimplicial set
\[
Y_{\bullet, n,k}(M,\phi):=B_{n}(B_{\bullet}(M,\phi)_k,S_{k}(\tH_{0,\partial}(M)),*)\xrightarrow{\epsilon} B_{n}(*,S_{k}(\tH_{0,\partial}(M)),*).
\] 
We denote the realization of $Y_{p, n,k}(M,\phi)$ in the bisimplicial directions by $Y_p(M,\phi)$. If we realize in the semisimplicial direction, we obtain the map
\[
||\epsilon||\colon |Y_{\bullet}(M,\phi)|\to \mathrm{B}|S_{\bullet}(\tH_{0,\partial}(M))|.
\]
\end{defn}
\begin{defn}\label{defn4} Similar to \Cref{defn2}, we define the $1$-handle  resolution associated to $\phi$ for $\BdH_{0,\partial}(M)$ is the augmented semisimplicial space
\[
\theta\colon Y_{\bullet}^{\delta}(M,\phi):= B_{\bullet}^{\delta}(M,\phi)\hcoker \tdH_{0,\partial}(M)\to \BdH_{0,\partial}(M).
\]
If we realize in the semisimplicial direction, we obtain the map
\[
|\theta|: |Y_{\bullet}^{\delta}(M,\phi)|\to \BdH_{0,\partial}(M).
\]
To show that the maps $||\epsilon||$ and $|\theta|$ are weak equivalences as before, we need to show that $|B^{\delta}_{\bullet}(M,\phi)|$ and $|B_{\bullet}(M,\phi)|$ are weakly contractible. 
\end{defn}
\begin{lem}\label{arc}
The realizations $|B^{\delta}_{\bullet}(M,\phi)|$ and $|B_{\bullet}(M,\phi)|$ are weakly contractible. 
\end{lem}
\begin{proof}
It is enough to show that $|B^{\delta}_{\bullet}(M,\phi)|$ is contractible since the weak contractibility of $|B_{\bullet}(M,\phi)|$ is deduced from that of $|B^{\delta}_{\bullet}(M,\phi)|$ by the same argument as in \Cref{claim1}.  To show that the simplicial complex $|B^{\delta}_{\bullet}(M,\phi)|$ is contractible, we prove that  all continuous maps $f:S^k\to |B^{\delta}_{\bullet}(M,\phi)|$ are nullhomotopic for all $k$. Without loss of generality, we can assume that $f$ is a PL map with respect to a triangulation $K$ of $S^k$. To show that $f$ is nullhomotopic, we show that there exists $\alpha\in  B^{\delta}_{\bullet}(M,\phi)$ so that one can change $f$ up to homotopy so that  the image $f(K)$ lies in $\text{Star}(\alpha)$.

First we argue that we can assume the vertices of $f(K)$ are pairwise transverse after changing $f$ up to homotopy. For transversality in the topological category, we need the data of the normal microbundle (see \cite[Essay 3, section 1]{MR0645390}), but by the Schoenflies theorem all embedded arcs are bicollared so in this dimension, the embedded arcs have a unique normal microbundle data. Therefore, we did not need to add the microbundle data (germ of the cocore around the core of the handle) to the definition of $B^{\delta}_{\bullet}(M,\phi)$.

\noindent{\bf Claim:} After changing $f$ up to homotopy, we can assume that all vertices in $f(K)$ as arcs in $M$ are pairwise transverse. 

Doing this in the smooth category is easier, because transversality is an open condition in $C^{\infty}$-topology and one can change an arc up to small isotopy so that it becomes simultaneously transverse to several other arcs. But in the topological category one needs to do it inductively. We give an argument in a way that can be generalized to the higher dimensions. 

\noindent{\it Proof of the claim:}
Since the arcs are bicollared, by a parallel copy of an arc $\psi$ we mean an embedded arc close to $\psi$ in its collar neighborhood that is disjoint from $\psi$ and satisfies equation (\ref{cond1}) for some $t$. Let us enumerate   the vertices of $f(K)$ by $\psi_1, \psi_2,\dots, \psi_m$. First we choose a parallel copy of $\psi_2$ and perturb it by a small isotopy to obtain $\psi'_2$ so that it  becomes transverse to  $\psi_1$. If the isotopy is small enough  $\psi'_2$ is disjoint from  $\psi_2$ and  all vertices in $f(K)$ that $\psi_2$ was disjoint from. Therefore, there is a homotopy replacing $\psi_2$ with $\psi'_2$ and fixing the image of other vertices. Thus we may assume that $\psi_1$ and $\psi_2$ are transverse.  Hence their intersection  is a set of points. 

Now we move on to $\psi_3$. Similarly by choosing a nearby copy of $\psi_3$ and a small perturbation $\psi'_3$ of this nearby copy, we obtain an arc that is disjoint from the points in the intersection of the previous two arcs $\psi_1$ and $\psi_2$.  Hence we can choose a small neighborhood $U$ of the intersection of  $\psi_1$ and $\psi_2$ such that  $\psi'_3$ is also disjoint from $U$. Now by Quinn's transversality (\cite{quinn1988topological}), we can find a small isotopy whose support is away from $U$ and we obtain an arc $\psi''_3$ that is transverse to the submanifold $(\psi_1(D^1)\cup \psi_2(D^1))\backslash U$. If we choose the isotopy small enough the arc $\psi''_3$ is disjoint from  $\psi_3$ and all the vertices of $f(K)$ that $\psi_3$ was disjoint from. Hence by a homotopy of the map $f$, we can replace $\psi_3$ with $\psi''_3$. Therefore, we may assume that $\psi_3$ is transverse to  $\psi_1$ and $\psi_2$. By continuing this process we can change $f$ up to homotopy to make  $\psi_i$ transverse to $\psi_j$ for $j<i$. Hence, we may assume that  the vertices of $f(K)$ are pairwise transverse to each other and the $t$-coordinates are different. $\blacksquare$

Similarly, we choose a vertex $\psi\in B^{\delta}_{\bullet}(M,\phi)$ that is transverse to all the vertices in $f(K)$. If a vertex $v\in f(K)$ intersects the arc $\psi$, because of the condition (\ref{cond1}) it intersects in even number of points $\{p_1, p_2, \dots, p_{2i_v}\}$. So the consecutive points of the intersection $\{p_{2i-1}, p_{2i}\}$ give union of disjoint subintervals on $\psi$. Let $I_v$ denote these subintervals on $\psi$ associated to its intersection with $v$.  The simplicial complex $f(K)$ has finitely many vertices. From the set of intervals $\{I_v\}_{v\in f(K)}$ on $\psi$, we choose a maximal family of  subintervals that are either disjoint or one includes the other. Such family has an innermost subinterval $D$ on the arc $\psi$. Suppose that this innermost subinterval is in $I_{\psi_1}$. Since the arcs in $\text{Star}(\psi_1)$ are disjoint from $\psi_1$ and $D$ is an innermost subinterval in a maximal family, $D$ is disjoint from the arcs in $\text{Star}(\psi_1)$. The two ends $\partial D$ on $\psi$ also lies on $\psi_1$ and they bound a subinterval $D'$ on $\psi_1$. 

 Since $\psi_1$ is isotopic to a parallel copy of $\psi$, by \cite[Proposition 1.7]{farb2011primer} there is a Whitney disk $N$ (or bigon in the context of surgery of arcs on surfaces) that bounds $D\cup D'$. Hence, by doing the Whitney trick, we can choose an arc $\psi'_1$ so that it is disjoint from $\psi_1$ and all arcs in $\text{Star}(\psi_1)$ and also it intersects $\psi$ in fewer points. Therefore, there is a simplicial homotopy of $f$ that replaces $\psi_1$ with $\psi'_1$. By continuing this process, we can homotope $f$ to the star of the arc $\psi$. Hence, $f$ is nullhomotopic.
\end{proof}

\begin{cor}\label{cor1}
The maps $||\epsilon||$ and $|\theta|$ in \Cref{defn3} and \Cref{defn4} are weak homotopy equivalence. 
\end{cor}
Note that unlike the $0$-handle case, we have a semisimplicial map 
\[
 B^{\delta}_{\bullet}(M,\phi)\to  B_{\bullet}(M,\phi)_{\bullet},
\]
which is the inclusion to the $0$-simplices in the simplicial direction and is equivariant with respect to the homomorphism $\tdH_{0,\partial}(M)\to S_{\bullet}(\tH_{0,\partial}(M))$. Therefore, we obtain a semisimplicial map between $1$-handle resolutions
\begin{equation}\label{eq:7}
\alpha_{\bullet}\colon Y_{\bullet}^{\delta}(M,\phi)\to Y_{\bullet}(M,\phi).
\end{equation}
By comparing the spectral sequences for $Y_{\bullet}^{\delta}(M,\phi)$ and $Y_{\bullet}(M,\phi)$, we want to show that \Cref{main} holds true for $M$ if it is true for a surface $M\backslash \phi$ that is obtained from $M$ by removing the $1$-handle $\phi$. 

\begin{defn}For a $p$-simplex $\sigma_p\in B^{\delta}_{\bullet}(M,\phi)$. Let  $M\backslash \sigma_p$ denote the surface  that we obtain by cutting $M$ along the arcs in $\sigma_p$.
\end{defn}
\begin{prop}\label{reduction2}Suppose \Cref{main} holds for $M\backslash \phi$. Then it also holds for $M$. 
\end{prop}
\begin{proof}The map $\alpha_{\bullet}$ in (\ref{eq:7}) induces a map of spectral sequences
\begin{equation}
\begin{gathered}
\begin{tikzcd}
H_q(Y_{p}^{\delta}(M,\phi))\arrow["\alpha_*"]{r}\arrow[Rightarrow]{d}&H_q(Y_{p}(M,\phi))\arrow[Rightarrow]{d}\\H_{p+q}(|Y_{\bullet}^{\delta}(M,\phi)|)\arrow[""]{r}\arrow["\cong"]{d}& H_{p+q}(|Y_{\bullet}(M,\phi)|)\arrow["\cong"]{d}\\ H_{p+q}(\BdH_{0,\partial}(M))\arrow[""]{r}& H_{p+q}(\mathrm{B}|S_{\bullet}(\tH_{0,\partial}(M))|).
\end{tikzcd}
\end{gathered}
\end{equation}
So if we show that the hypothesis implies that $\alpha_*$ is an isomorphism, then we have the isomorphism on the $E^{\infty}$-page which concludes \Cref{main} for $M$. 

First, we shall observe that  that the set of connected components of $Y_{p}^{\delta}(M,\phi)$ is the same as that of $Y_{p}(M,\phi)$. Note that the former set is the same as the set of the orbits of the action of $\tdH_{0,\partial}(M)$ on $B^{\delta}_{p}(M,\phi)$\footnote{In fact, as we shall also use it in \Cref{surface}, the orbit of a $p$-simplex $$((t_0,f_0), (t_1, f_1),\dots, (t_p, f_p))$$ is uniquely determined by the real numbers $t_i$.}. Hence, by the isotopy extension theorem, this set of orbits can identified with $\pi_0(B^{\bf t}_{p}(M,\phi))$ which is the isotopy classes of $p$-simplices relative to the boundary. On the other hand, similar to \Cref{lem2}, we have a map
\[
Y_{p}(M,\phi)\to B^{\bf t}_{p}(M,\phi)\hcoker \tH_{0,\partial}(M),
\]
which is a weak equivalence. Therefore, they have isomorphic set of connected components. But since $\pi_1(\mathrm{B}\tH_{0,\partial}(M))=0$, the long exact sequence for the Borel construction implies that 
\[
\pi_0(B^{\bf t}_{p}(M,\phi))\xrightarrow{\cong}\pi_0(B^{\bf t}_{p}(M,\phi)\hcoker \tH_{0,\partial}(M)).
\]
Therefore, we have $\pi_0(Y_{p}^{\delta}(M,\phi))\cong \pi_0(Y_{p}(M,\phi))$.

For a $p$-simplex $\sigma_p\in B^{\bf t}_{p}(M,\phi)$, let $\text{Stab}(\sigma_p)$ be its stabilizer as a subgroup of $\tH_{0,\partial}(M)$ and let $\text{Stab}^{\delta}(\sigma_p)$ denote the same group with the discrete topology. Similar to \Cref{lem1} and \Cref{lem2}, we have a homotopy commutative diagram
\begin{equation}\label{e}
\begin{gathered}
\begin{tikzcd}
Y_p^{\delta}(M,\phi)\arrow[""]{r}&Y_{p}(M,\phi) \\ \coprod_{\text{orbits}}\mathrm{B}\text{\textnormal{Stab}}^{\delta}(\sigma_p)\arrow["\simeq"]{u}\arrow[""]{r}& \coprod_{\text{orbits}}\mathrm{B}|S_{\bullet}(\text{\textnormal{Stab}}(\sigma_p))|\arrow["\simeq"]{u}.
\end{tikzcd}
 \end{gathered}
 \end{equation}
 Note that if we remove the arcs in $\sigma_p$ from $M$, we obtain a surface with $p$ components such that $p-1$ of them are homeomorphic to disks and one of them is homeomorphic to $M\backslash \phi$. Therefore the hypothesis and  Mather's theorem (\cite{MR0288777}) for disks imply that the map
 \[
 \BdH_{0,\partial}(M\backslash \sigma_p)\to \BH_{0,\partial}(M\backslash \sigma_p),
 \]
 induces a homology isomorphism for all $p$ and all $\sigma_p$. Since $\tH_{0,\partial}(M\backslash \sigma_p)$ is the identity component of $\text{\textnormal{Stab}}(\sigma_p)$, the comparison of fibration similar to the diagram (\ref{ee}) implies that the map 
 \[
 \mathrm{B}\text{\textnormal{Stab}}^{\delta}(\sigma_p)\to \mathrm{B}\text{\textnormal{Stab}}(\sigma_p),
 \]
 induces a homology isomorphism. Therefore, $\alpha_*$ is an isomorphism. 
\end{proof}

\begin{proof}[Proof \Cref{main} for surfaces] Using \Cref{reduction1}, we know that in order to prove \Cref{main} for a closed surface $\Sigma_g$ of genus $g$, it is enough to prove it for a surface $\Sigma_{g,k}$ of genus $g$ and $k$ boundary components for all positive integer $k$. We induct on $-\chi(\Sigma_{g,k})=2g+k-2$. The base case is when $g=0$ and $k=1$ which is homeomorphic to a disk and is given by Mather's theorem (\cite{MR0288777}). In general, if $g>0$, we can choose the $1$-handle $\phi$ so that its two ends lie on the same boundary component and $\Sigma_{g,k}\backslash \phi$ is homeomorphic to $\Sigma_{g-1,k+1}$. Since $-\chi(\Sigma_{g-1,k+1})<-\chi(\Sigma_{g,k})$, our induction hypothesis and \Cref{reduction2} imply that \Cref{main} holds for $\Sigma_{g,k}$. 

So now we suppose that $g=0$ and $k>1$. To reduce the number of boundary components, we can choose the $1$-handle $\phi$ so that its two ends lie on different boundary components and $\Sigma_{0,k}\backslash \phi$ is homeomorphic to $\Sigma_{0,k-1}$. Since $-\chi(\Sigma_{0,k-1})<-\chi(\Sigma_{0,k})$, the induction hypothesis and \Cref{reduction2} imply that \Cref{main} also holds for $\Sigma_{0,k}$ for all $k$. Therefore, it holds for $\Sigma_{g,k}$ for all $g$ and $k$. 
\end{proof}
\begin{rem}
In fact using handle resolutions whose cores have dimensions less than half of the dimension of $M$, one can show that Thurston's theorem for a manifold $M$ whose dimension is larger than $4$ is equivalent to Thurston's theorem for a trivial bordism $N\times D^1$ where $N$ is a manifold whose dimension is $\text{dim}(M)-1$. But it is not known to the author whether for handle of dimension $\text{dim}(M)/2$ or higher, a similar contractibility statement as \Cref{arc} holds.
\end{rem}
\section{Cutting three manifolds into disks} \label{balls}To do exactly similar argument as the case of surfaces, we need to find contractible semi-simplicial spaces that cut the manifold into union of $3$-disks. In this section, disks are $3$-dimensional unless mentioned otherwise. Doing an inductive process to cut a three manifold into disks, however, is harder than the case of surfaces. 

For certain types of three manifolds,  namely for Haken $3$-manifolds, this process of cutting into disks is well known. Recall that $M$ is Haken if it is irreducible and contains a properly embedded two sided incompressible surface. Being an irreducible $3$-manifold means that every embedded $2$-sphere bounds a disk. The existence of this disk allows us to do a similar surgery argument as we did for isotopic arcs in a surface. Recall that a compact connected surface $S$ that is not homeomorphic $S^2$, in $M$ is an incompressible surface, if it is properly embedded $S\cap \partial  M=\partial S$, and the normal bundle of $S$ is trivial and the inclusion $S\hookrightarrow M$ is $\pi_1$ injective. Given the Haken manifold theory, there is a finite sequence of incompressible surfaces so that as we cut a Haken manifold $M$ along those surfaces, we obtain disjoint union of disks. 

The idea is to  induct on the number of prime factors in a prime decomposition of $M$ to reduce Thurston's theorem to the case of Haken manifolds and then use the hierarchy of Haken manifolds to reduce it to the case of disks. 

Let $M\cong P_1\# P_2\#\cdots\# P_n$ be the connected sum of $n$ prime $3$-manifolds. We will define semi-simplicial spaces with contractible realizations that encode different ways of cutting $M$ into the union of its prime factors with certain number of disks removed. By the same argument as the previous section, a spectral sequence argument shows that Thurston's theorem holds for $M$ if it does for manifolds homeomorphic to $P_i$ with certain number of disks removed. We then reduce Thurston's theorem for  such manifolds to the case of the Haken manifolds.

\subsection{Cutting along  separating spheres} We want reduce \Cref{main} for $M$ to $3$-manifolds with fewer prime factors. To do so, we shall define semi-simplicial simplicial sets parametrizing {\it separating spheres} in $M$. By a separating sphere, we mean an embedded sphere that does not bound a disk in $M$.  If $M$ has a sphere boundary, a separating sphere could be isotopic to a sphere boundary component.
\begin{defn}
Let $S(M)$ be a simplicial complex whose set of vertices of $S(M)$ is given by locally flat embeddings of a $2$-sphere $\phi\in \text{Emb}^{\text{lf}}(S^2, M)$ so that its image is a separating sphere. A set of $p+1$ such embeddings constitutes a $p$-simplex if their images are disjoint.
\end{defn}
\begin{defn} We use $S(M)$ to define a semisimplicial set $S_{\bullet}^{\delta}(M)$ and a semisimplicial simplicial set $S_{\bullet}(M)_{\bullet}$ on which $\tH_{0,\partial}(M)$ and $S_{\bullet}(\tH_{0,\partial}(M))$ act respectively.
\begin{itemize}
\item\noindent{\bf Discrete version:} Let $S_0^{\delta}(M)$ be the set of the vertices of $S(M)$ and  the set of the $p$-simplices $S_p^{\delta}(M)$ be all different ways of ordering the $p$-simplices in $S(M)$. In other words, $S_p^{\delta}(M)$ is the set of $(p+1)$-tuples $(v_0,v_1,\dots,v_p)$ so that the set $\{v_0,v_1,\dots,v_p\}$ is a simplex in $S(M)$. 

\item\noindent{\bf Topological version:} Let $S_{\bullet}(M)_{\bullet}$ be a semisimplicial simplicial set whose $k$-simplices $S_{\bullet}(M)_{k}$ in the simplicial direction is given by  tuples of vertices $(v_0(t), v_1(t),\dots, v_{\bullet}(t))$ such that $\{v_0(t),v_1(t),\dots,v_{\bullet}(t)\}$ is a simplex in $S(M)$ for all $t\in \Delta^k$. 
\end{itemize}
\end{defn}
We shall first prove that $S(M)$ is contractible when it is non-empty and then deduce that realizations of  $S_{\bullet}^{\delta}(M)$ and  $S_{\bullet}(M)_{\bullet}$ are contractible.
\begin{lem}\label{claim2}
If $M$ is not a prime manifold, the simplicial complex $S(M)$ is contractible.
\end{lem}
\begin{proof}
Similar to \Cref{claim1}, we want to show that for all $k$, any continuous map $f\colon S^k \to S(M)$ is nullhomotopic.  Without loss of generality, we can assume that for  a triangulation $K$ of $S^k$,  the map $f$ is PL. To find a nullhomotopy for the map $f$, it is enough to homotope it so that its image lies in the star of a vertex in $S(M)$.

As the claim in proof of \Cref{arc}, we can homotope $f$ so that  the vertices in $f(K)$ are pairwise transverse. Let $v_1\in S(M)$ be an embedding whose image is transverse to the spheres represented by the set of vertices in $f(K)$.  To homotope $f$ so that its image lies in $ \text{Star}(v_1)$, we inductively remove the circles in the intersection of  $v_1$ and the spheres in $f(K)$. 
\begin{figure}[h] 

 \begin{tikzpicture}[node distance=6cm]
 \draw [line width=1.05pt] (0,0) circle (2cm);
  \draw [gray, line width=1.05pt] (2,0) circle (1.7cm);
\draw [dashed, line width=1.05pt] (0.8,1.4) arc (60:300:1.6cm);
\draw [dashed, line width=1.05pt] (0.8,1.4) arc (120:240:1.6cm);
\draw [dashed, line width=1.05pt] (1.2,1.2) arc (120:240:1.4cm);
\draw [dashed, line width=1.05pt] (1.2,1.2) arc (60:-60:1.4cm);
\node at (-2, 1.5) {$S$};
\node at (-0.6, 1) {$S'$};
\node at (1.3, 0.6) {$S''$};
\node at (3.8, 1.3) {$S_1$};
 \end{tikzpicture}
 \caption{Surgery on spheres in one dimension lower}\label{surgery} \end{figure} 

Let $S_1$ be the  embedded sphere given by the image of $v_1$. The intersection of the spheres  in $f(K)$ and $S_1$ form a finite number of circles. Among these circles, we choose a maximal family of disjoint circles on $S_1$. Let $C$ be an innermost circle in this family, which is given by the intersection of $S_1$ and a sphere $S$ given by the image of an embedding $f(x)=v\in f(K)$. The circle $C$ bounds a $2$-disk in $S_1$. We can cut $S$ along the circle $C$  and glue two copies of this $2$-disk to obtain two disjoint embedded spheres $S'$ and $S''$ (see \Cref{surgery}). By considering nearby parallel copies, we can assume that $S$, $S'$ and $S''$ are disjoint. Note that at least one of the spheres $S'$ and $S''$ is separating. We assume that $S'$ is separating. Now we shall replace $S$ by $S'$ as the image of the vertex $v$ as follows  (see \Cref{fig1}). We choose an embedding $v'$ whose image is $S'$. By choosing nearby parallel copies of the spheres, we can assume that the vertex $v'$ is connected to $v$ i.e. their corresponding spheres are disjoint.
\begin{figure}[h]
\[
\begin{tikzpicture}[scale=.3]

%\draw[line width=1.05pt] (0,0) circle (5);
\draw[line width=1.05pt]  (5,0) arc (0:45: 5);
\draw[line width=1.05pt]  (1.29,4.83) arc (75:105: 5);

%\draw[line width=1.05pt]  (5,0) arc (55:100: 5);
\draw[line width=1.05pt]  (5,0) arc (0:-108: 5);
\draw[line width=1.05pt]  (-5,0) arc (180:225: 5);
\draw[line width=1.05pt]  (-5,0) arc (180:135: 5);

\draw[line width=1.05pt] [dashed] (5,0) arc (10: 170: 5 and 1);
\draw[line width=1.05pt]  (5,0) arc (-10: -170: 5 and 1);

\draw[line width=1.05pt] (3,2.5).. controls (2,2) and (1.5, 2.5).. (1,3.2);
\draw[line width=1.05pt] [dashed] (3,2.5).. controls (2,3.2) and (1.5, 3.8).. (1,3.2);

\draw[line width=1.05pt] (-3,2.5).. controls (-2,2) and (-1.5, 2.5).. (-1,3.2);
\draw[line width=1.05pt] [dashed] (-3,2.5).. controls (-2,3.2) and (-1.5, 3.8).. (-1,3.2);

\draw[line width=1.05pt]  (-3,-2.5).. controls (-2,-2) and (-1.5, -2.5).. (-1,-3.2);
\draw[line width=1.05pt] [dashed] (-3,-2.5).. controls (-2,-3.2) and (-1.5, -3.8).. (-1,-3.2);

\draw [line width=1.05pt] (3,2.5) .. controls (7,10) and (4,12)  .. (1,3.2);
\draw [line width=1.05pt] (-3,2.5) .. controls (-7,10) and (-4,12)  .. (-1,3.2);
\draw [line width=1.05pt] (-3,-2.5) .. controls (-7,-10) and (-4,-12)  .. (-1,-3.2);

\begin{scope}[shift={(-6.8,7)}, scale=0.5]
\draw [line width=1.05pt] (4.6,.5) arc (240:300:2.5 and 6.75);
\draw   [line width=1.7pt] (6.87, 0.13) arc (48:150:1.2 and 0.9);
\end{scope}
\begin{scope}[shift={(-6.5,6)}, scale=0.5]
\draw [line width=1.05pt] (4.6,.5) arc (240:300:2.5 and 6.75);
\draw   [line width=1.7pt] (6.87, 0.13) arc (48:150:1.2 and 0.9);
\end{scope}
\begin{scope}[shift={(-6.1,5)}, scale=0.5]
\draw [line width=1.05pt] (4.6,.5) arc (240:300:2.5 and 6.75);
\draw   [line width=1.7pt] (6.87, 0.13) arc (48:150:1.2 and 0.9);
\end{scope}

\begin{scope}[shift={(0.3,6)}, scale=0.5]
\draw [line width=1.05pt] (4.6,.5) arc (240:300:2.5 and 6.75);
\draw   [line width=1.7pt] (6.87, 0.13) arc (48:150:1.2 and 0.9);
\end{scope}

\begin{scope}[shift={(-6.8,-7)}, scale=0.5]
\draw [line width=1.05pt] (4.6,.5) arc (240:300:2.5 and 6.75);
\draw   [line width=1.7pt] (6.87, 0.13) arc (48:150:1.2 and 0.9);
\end{scope}
\begin{scope}[shift={(-6.5,-6)}, scale=0.5]
\draw [line width=1.05pt] (4.6,.5) arc (240:300:2.5 and 6.75);
\draw   [line width=1.7pt] (6.87, 0.13) arc (48:150:1.2 and 0.9);
\end{scope}

\node at (2,1.5) {$S$};
\node at (-2,1.5) {$S'$};
\node at (-2,-1.5) {$S''$};
\end{tikzpicture}
\]
\caption{Separating spheres $S$, $S'$ and $S''$ are depicted in one dimension lower. They bound a $3$-sphere with three disks removed.}\label{fig1}
\end{figure}

If we choose $S'$ sufficiently close to the $2$-disk in $S$ that bounds $C$, then any sphere $S_2$ in the star of $v$ which intersected $S'$ would also intersect this $2$-disk.  However, this cannot happen: since $S_2 \cap S_1 = \emptyset$, and $C$ was chosen to be an innermost circle among a maximal family of disjoint circles given by intersections with $S_1$, no disjoint sphere $S_2$ can intersect the $2$-disk bounded by $C$.  Thus, no vertex in the star of $v$ intersects $S'$, so our modified sphere as the image of $v'$ 
 remains disjoint from all the spheres that the image of $v$ is disjoint from.  In other words, $v'$ is connected to all vertices in the star of $v$. Therefore, we have a simplicial homotopy $F\colon K\times [0,1]\to S(M)$ such that $F(-,1)$ is the same as $F(-,0)$ on all vertices but $x$ and $F(x,1)=v'$. Note that the vertices in the image $F(-,1)\colon K\to S(M)$  have fewer circles in their intersection with $S_1$. By repeating this process, we could homotope the map $f$ to a map whose image lies in the star of $v_1$. Therefore, $f$ is nullhomotopic.
\end{proof}
To prove that realizations of  $S_{\bullet}^{\delta}(M)$ and  $S_{\bullet}(M)_{\bullet}$ are contractible, we need $S(M)$ to have a property that is called {\it weakly Cohen-Macaulay}.

\begin{defn}  A simplicial complex $K$ is called {\it weakly Cohen-Macaulay} of dimension at least $n$ and it is denoted by  $\it{w}CM(K)\geq n$ if it is $(n-1)$-connected and the link of any $p$-simplex is $(n-p-2)$-connected. %It is called {\it locally weakly Cohen-Macaulay} if the link of any $p$-simplex is $(n-p-2)$-connected with no global connectivity assumption.
 \end{defn}
For  a simplex $\sigma$ in the simplicial complex $S(M)$, let $M\backslash \sigma$ denote a manifold that is homeomorphic to the manifold obtained from $M$ by cutting it along the spheres in $\sigma$. Note that a link of $\sigma$ is homeomorphic to $S(M\backslash \sigma)$ which is again contractible by \Cref{claim2}. Therefore $S(M)$ is a weakly Cohen-Macaulay of dimension infinity.
\begin{prop}\label{CM}
If $M$ is not a prime manifold, the realizations of  $S_{\bullet}^{\delta}(M)$ and  $S_{\bullet}(M)_{\bullet}$ are contractible.
\end{prop}
\begin{proof}
Similar to \Cref{claim1}, it is enough to show that the realization of the semisimplicial set $S_{\bullet}^{\delta}(M)$ is contractible. But $S_{\bullet}^{\delta}(M)$ is obtained from $S(M)$ by considering all different orderings on simplices. It is a consequence of {\it generalized coloring lemma} (\cite[Theorem 2.4]{galatius2014homological}) that if $wCM(S(M))>n$ then $|S_{\bullet}^{\delta}(M)|$ is at least $(n-1)$-connected (see \cite[Theorem 3.9]{nariman2014homologicalstability} for a similar argument). Therefore, $|S_{\bullet}^{\delta}(M)|$ is contractible. 
\end{proof}
As the case of surfaces, we can use $S_{\bullet}^{\delta}(M)$ and  $S_{\bullet}(M)_{\bullet}$ to define semsimplicial resolutions for $\BdH_{0,\partial}(M)$ and $\mathrm{B}|S_{\bullet}(\tH_{0,\partial}(M))|$. Therefore, similar spectral sequence arguments as \Cref{reduction1} and \Cref{reduction2} imply the following reduction of \Cref{main}.
\begin{prop}\label{reduction3}
If \Cref{main} holds for $M\backslash \sigma$ for all simplices $\sigma$ in $S(M)$, then it also holds for $M$.
\end{prop}
By the uniqueness of the prime decomposition for three manifolds, it is easy to see that $M\backslash \sigma$ is a disjoint union of pieces that are homeomorphic  to either one of the $P_i$'s with certain number of disks removed or $S^3$ with certain number of disks removed. Hence, by repeating \Cref{reduction3} we conclude that having the Thurston theorem for  prime manifolds or $S^3$ with a nonzero number of disks removed, implies Thurston's theorem for all three manifolds that are not prime.

\subsection{Reducing  to the case of Haken manifolds} If $M$ is not prime, by cutting $M$ along sphere systems, we reduced to the case of prime manifolds with a number of disks removed and $S^3$ with a number of disks removed. To reduce these cases further to the case of the Haken manifold, we remove $1$-handles from these pieces. If $M$ is prime, we remove solid tori to reduce it to the case irreducible manifolds with torus boundary components which are known to be Haken.

First, we shall treat the case where $M$ is prime. Recall that the only prime manifold that is not irreducible is $S^1\times S^2$ (see \cite[Proposition 1.4]{hatcher2000notes}).  To remove solid tori, similar to the case of $0$-handles, we need to consider the germs of embeddings around the core of the solid torus.
\begin{defn}Let $f\colon S^1\times D^2\hookrightarrow M$ and $g\colon S^1\times D^2\hookrightarrow M$ be two locally flat embeddings. We say they have the same germ around the cores $f(S^1\times \{0\})$ and $g(S^1\times \{0\})$ if there exists some open neighborhood $U\subset D^2$ of the origin such that $f|_{S^1\times U}=g|_{S^1\times U}$. We shall denote the germ of $f$ around its core by $[f]$.
\end{defn}
\begin{defn}Let $\phi: S^1\times D^2\hookrightarrow M$ be a $\pi_1$-injective embedding. Let $T(M;\phi)$ be a simplicial complex whose vertices are germ of embeddings $f\colon S^1\times D^2\hookrightarrow M$ such that the core of $f$ is isotopic to the core of $\phi$. And $\{[f_0], [f_1], \dots, [f_p]\}$ constitutes a $p$-simplex if the cores of $[f_i]$ are disjoint. 
\end{defn}
\begin{lem}
The simplicial complex $T(M;\phi)$ is weakly Cohen-Macaulay of dimension infinity.
\end{lem}
\begin{proof}
We first show that $T(M;\phi)$ is contractible and then it becomes clear that the links of a simplex is contractible by the same argument. As before, we want to show that for all $k$, any continuous map $f\colon S^k\to T(M;\phi)$ is nullhomotopic. Without loss of generality, we can assume that $f$ is a PL map with respect to a triangulation $K$ on $S^k$. 

Note that since the codimension of the core of an embedded solid torus in $M$ is $2$, if two cores are transverse in $M$, they should be disjoint. Since disjointness is an open condition, we can change $f$ up to homotopy so that all vertices of $f(K)$ are disjoint. We choose another vertex $v\in  T(M;\phi)$ whose core is disjoint from the cores of the vertices in $f(K)$. Therefore, $f(K)\subset \text{Star}(v)$ which implies that $f$ is nullhomotopic. 
\end{proof}
\begin{defn} We define a semisimplicial set $T^{\delta}_{\bullet}(M;\phi)$ and a semisimplicial simplicial set $T_{\bullet}(M;\phi)_{\bullet}$ on which $\tH_{0,\partial}(M)$ and $S_{\bullet}(\tH_{0,\partial}(M))$ act respectively.
\begin{itemize}
\item\noindent{\bf Discrete version:} The set of the $p$-simpleces of  $T^{\delta}_{p}(M;\phi)$ is given by the set of all different orderings on the set of $p$-simplices  of the simplicial complex $T(M;\phi)$. In other words, the $(p+1)$-tuple $([f_0], [f_1], \dots, [f_p])$ of germs of embeddings of solid tori is a $p$-simplex if the cores of $[f_i]$'s are disjoint.  The $i$-th face maps is given by omitting $[f_i]$. 
\item{\bf Topological version} The set of $0$-simplices  $T_{0}(M;\phi)_{\bullet}$ in the semisimplicial direction is the subset of embeddings $f(t)$ in $\text{Emb}_{\bullet}^{\text{lf}}(S^1\times D^2, M)$ whose core $f(t)(S^1\times\{0\})$  is isotopic to the core of $\phi$ for all $t\in \Delta^{\bullet}$.  A $p$-simplex in the semisimplicial direction is given by $(p+1)$-tuples $(f_0(t), f_1(t), \dots, f_p(t))$ in $T_{0}(M;\phi)_{\bullet}^{p+1}$ so that $f_i(t)(S^1\times D^2)$ and $f_j(t)(S^1\times D^2)$ are disjoint for all $i,j$ and $t\in \Delta^{\bullet}$. The face maps in both directions are defined as usual. 
\end{itemize}
\end{defn}
\begin{prop}
The realizations of $T^{\delta}_{\bullet}(M;\phi)$ and $T_{\bullet}(M;\phi)_{\bullet}$ are contractible
\end{prop}
\begin{proof}
It is similar to \Cref{CM}.
\end{proof}
 For a $p$-simplex $\sigma$ in $T_{p}(M;\phi)_0$, let $M\backslash \sigma$ denote a manifold obtained from $M$ by cutting the interior of the solid tori in $\sigma$.
\begin{prop}\label{reduction4} If \Cref{main} holds for $M\backslash \sigma$ for all simplices $\sigma$ in $T_{p}(M;\phi)_0$, then it also holds for $M$. 
\end{prop}
\begin{proof}
Similar to the diagram \ref{e'} in \Cref{disk}, we can define a zig-zag of maps between $T^{\delta}_{\bullet}(M;\phi)$ and $T_{\bullet}(M;\phi)_{\bullet}$ that would lead to a map of spectral sequences as in \Cref{reduction1}. The hypothesis implies the isomorphism on the first page of the spectral sequences. Hence, as in \Cref{reduction1}, it implies \Cref{main} for $M$. 
\end{proof}
\begin{rem}\label{1handle}Note that the only thing we used about the solid torus $\phi$ was the dimension of its core is less than its codimension. Similarly, if  $M$ has boundary, we can define a $1$-handle $\phi\colon D^1\times D^2\to M$ so that its two ends $\phi(\{0\}\times D^2)$ and $\phi(\{1\}\times D^2)$ lie on the boundary of $M$. And then we could define $T^{\delta}_{\bullet}(M;\phi)$ and  $T_{\bullet}(M;\phi)_{\bullet}$ to reduce \Cref{main} to the case of $M$ with certain $1$-handles removed.
\end{rem}
Now we are ready to use \Cref{reduction4}, to treat the case when $M$ is prime. 
\begin{prop}\label{reduction5}
 \Cref{Main} holds for closed prime manifolds, if it holds for all irreducible $3$-manifolds whose boundary components are homeomorphic to a torus.
\end{prop}
\begin{proof}
We choose a $\pi_1$-injective locally flat embedding $\phi\colon S^1\times D^2\hookrightarrow M$. If  $M$ is a closed irreducible $3$-manifold, then for any $p$-simplex $\sigma$ in $T_{p}(M;\phi)_0$, the manifold $M\backslash \sigma$ is still irreducible. Therefore, this case is easily followed from \Cref{reduction4}.

The only prime manifold that is not irreducible is $S^1\times S^2$. Since all the essential embedded spheres in $S^1\times S^2$ are isotopic, all the embedded spheres in the manifold $S^1\times S^2\backslash \sigma$ bound a disk. Hence, by \Cref{reduction4}, the case of $S^1\times S^2$ is also reduced to the case of  irreducible $3$-manifolds with torus boundary components.
\end{proof}
Therefore, if $M$ is prime, we can reduce \Cref{main} to the case of Haken manifolds because irreducible $3$-manifolds with torus boundary are Haken. If $M$ is not prime,  using \Cref{reduction3}, \Cref{reduction4}, and \Cref{reduction5}, we could reduce \Cref{main} for $3$-manifolds to the case of $S^3\backslash \cup_{i=1}^k\text{int}(D^3)$ and $P_i\backslash \cup_{i=1}^k\text{int}(D^3)$ where $P_i$'s are irreducible. The goal is to further reduce these cases to the case for Haken manifolds. As the general strategy is to cut along submanifolds, we always get manifolds with boundary. Furthermore, an irreducible $3$-manifold with boundary is Haken. So to reduce to the Haken manifolds, we want to cut along submanifolds to get an irreducible $3$-manifold with boundary.

\begin{prop}
\Cref{Main} holds for $S^3$ with a number of disks removed if it holds for Haken $3$-manifolds.
\end{prop}
\begin{proof}
Let $M=S^3\backslash \cup_{i=1}^k\text{int}(D^3)$. We can assume that $k>1$ because otherwise $M$ is homeomorphic to a $3$-disk and Thurston's theorem in this case is deduced from the Mather theorem (\cite{MR0288777}).  We choose a $1$-handle $\phi:D^1\times D^2\hookrightarrow M$ so that $\phi(\{0\}\times D^2)$ and $\phi(\{1\}\times D^2)$ are subsets of different sphere boundary components. Hence, by \Cref{1handle}, Thurston's theorem holds for $M$ if it holds for $M\backslash \sigma$ for all simplices $\sigma$ in $T_{\bullet}(M;\phi)_0$. But for a $p$-simplex $\sigma$, the manifold $M\backslash \sigma$ has fewer boundary components. By repeating this process, we can reduce the theorem for $M$ to a handlebody which has one boundary component. But a handlebody is Haken.
\end{proof}
Note that the sphere boundaries in $P_i\backslash \cup_{j=1}^{k}\text{int}(D^3)$ destroys the irreducibility. So to reduce Thurston's theorem for $P_i\backslash \cup_{j=1}^{k}\text{int}(D^3)$ to the case for Haken manifolds, we first cut along certain $1$-handles to reduce the number of sphere boundaries. But unlike the case of $S^3$ with a number of disks removed, we want to increase genus of each boundary component. Because it is not clear how to do the same procedure as for $S^3\backslash \cup_{i=1}^k\text{int}(D^3)$ to get a handlebody at the end, we first need to show that $P_i\backslash \cup_{j=1}^{k}\text{int}(D^3)$  are not simply connected.
\begin{lem}\label{poincare}
If a $3$-manifold $M$ with boundary is simply connected, it is obtained from $S^3$ by removing the interior of a union of disjoint  disks in $S^3$. 
\end{lem}
\begin{proof}
It is enough to show that the boundary $\partial M$ is homeomorphic to union of $S^2$'s. Because if we fill in the sphere boundaries by disks, we obtain a simply connected closed $3$-manifold which has to be homeomorphic to $S^3$ by Perelman's theorem (\cite{perelman2002entropy, perelman2003ricci}). Since $M$ is simply connected, we have $H_1(M)=0$, so by the Poincar\' e-Lefschetz duality, we also have $H_2(M,\partial M)\cong H^1(M)=0$. The homology long exact sequence for the pair $(M, \partial M)$ implies that $H_2(M,\partial M)\to H_1(\partial M)\to H_1(M)$ is exact. Therefore, $H_1(\partial M)=0$ which implies that $\partial M$ is homeomorphic to a union of $S^2$'s.
\end{proof}
Let $Q$ be the manifold obtained from $P$ by removing the interior of $m$ disjoint disks in $P$. To prove Thurston's theorem for $Q$, we want to cut $1$-handles from $Q$ to make it irreducible. Not that since $P$ is not simply connected and is not homeomorphic to sphere, so by \Cref{poincare}, the manifold $Q$ is not simply connected either. Let $\partial_iQ$ be the $i$-th boundary component. We choose  an arc $\gamma_i$ with the two ends on $\partial_iQ$ so that the arc $\gamma_i$ with a path between its two ends on the boundary is non-trivial in the fundamental group of $P$. 

Let $\phi_i: D^1\times D^2\hookrightarrow Q$ be a $1$-handle whose core is $\gamma_i$. Let us denote the manifold obtained from $Q$ by removing the interior of the handle $\phi_i$ by $Q\backslash \cup_{i=1}^m \phi_i$. 
\begin{lem}$Q\backslash \cup_{i=1}^m \phi_i$ is irreducible.
\end{lem}
\begin{proof}Given that $P$ is irreducible, every embedded sphere in $Q\backslash \cup_{i=1}^m \phi_i$ bounds a disk in $P$. If this disk contains any of the boundary components with a $1$-handle attached to it, then the core of the $1$-handle, union the path between the two ends of the core on the boundary, would be trivial in the fundamental group of $P$, which is a contradiction.
\end{proof}
\begin{prop}\label{irr}
\Cref{main} holds for $Q$, if it does for Haken manifolds.
\end{prop}
\begin{proof}
For any simplex $\sigma$ in $T_{\bullet}(M;\phi_i)_0$, the $i$-th boundary component of $Q\backslash \sigma$ is no longer sphere. Hence, by repeating \Cref{reduction4} for all boundary components, we could reduce \Cref{main} for $Q$ to the case of irreducible manifolds with boundary. 
\end{proof}

\subsection{Finishing \Cref{main}: the case of Haken $3$-manifolds with boundary}  By the theory of Haken manifolds (\cite{MR0160196}), we know that they have a hierarchy, where they can be split up into 3-disks along incompressible surfaces. Let $S$ be a surface with boundary and $\psi\colon S\times \bR\hookrightarrow M$ be a proper locally flat bicollared embedding of an incompressible surface. Given the case of Haken manifolds which are lower compared to $M$ in the Haken hierarchy, we inductively prove \Cref{main} for $M$ by cutting along incompressible surfaces.
\begin{defn} We define a semisimplicial set $K^{\delta}_{\bullet}(M,\psi)$ and a semisimplicial simplicial set $K_{\bullet}(M,\psi)_{\bullet}$ on which $\tH_{0,\partial}(M)$ and $S_{\bullet}(\tH_{0,\partial}(M))$ act respectively.
\begin{itemize}
\item\noindent{\bf Discrete version:} The $0$-simplices of $K^{\delta}_{\bullet}(M,\psi)$ are given by pairs $(t, \phi)$ such that $\phi\in \text{Emb}_0^{\text{lf}}(S,M)$ satisfying
\begin{equation}\label{bdry}
\phi(\partial S)=\psi(\partial S\times\{t\, \vec{e}\})),
\end{equation}
where $\vec{e}$ is the unit basis vector of $\bR$, and $\phi(S)$ is isotopic to the surface $\psi(S\times\{t\, \vec{e}\}))$.

The set of  $p$-simplices   $K^{\delta}_{p}(M,\psi)$, consists of $(p+1)$-tuples $$((t_0,\phi_0), (t_1,\phi_1),\dots, (t_p, \phi_p)),$$ in $K^{\delta}_{0}(M,\psi)^{p+1}$ so that $t_0<t_1<\cdots<t_p$ and  the embedded surfaces  $\phi_i(S)$ are disjoint. The face maps are given by forgetting the embeddings. Given that the $t$-coordinate in $(t,\phi)\in K^{\delta}_{0}(M,\psi)$ is uniquely determined by $\phi$, we shall write $\phi$ for a vertex and refer to its $t$-coordinate by $t_{\phi}$.  
\item\noindent{\bf Topological version:} The $0$-simplices $K_{\bullet}(M,\psi)_{0}$ in the simplicial direction is the same seimisimplicial set as $K^{\delta}_{\bullet}(M,\psi)$. Its $k$-simplices in the simplicial direction are given by tuples  $$((t_0,\phi_0(s)), (t_1,\phi_1(s)),\dots, (t_{\bullet}, \phi_{\bullet}(s))),$$ in $K^{\delta}_{0}(M,\psi)^{\bullet+1}$ for all $s\in \Delta^k$.

\end{itemize}
\end{defn}

\begin{prop}\label{cutincomp}
 Let $M$ be a Haken manifold with boundary. The realizations $|K^{\delta}_{\bullet}(M,\psi)|$ and $|K_{\bullet}(M,\psi)_{\bullet}|$ are contractible.
 \end{prop}  
 \begin{proof}
Similar to $\Cref{claim1}$, it is enough to show that $|K^{\delta}_{\bullet}(M,\psi)|$ is contractible. But since the $t$-coordinates of vertices of a simplex in $K^{\delta}_{\bullet}(M,\psi)$ is ordered, there is a natural order on simplices. Therefore $|K^{\delta}_{\bullet}(M,\psi)|$ is a simplicial complex. Let us represent an element of the homotopy group $f:S^k\to |K^{\delta}_{\bullet}(M,\psi)|$ by a PL map with respect to some triangulation $K$ on $S^k$. Similar to \Cref{arc} we shall arrange $f$ so that the vertices of $f(K)$ are pairwise transverse. 

Now we choose another vertex $\phi\in K^{\delta}_{0}(M,\psi)$ so that $\phi(S)$ is transverse to all vertices in $f(K)$ and its $t$-coordinate is different from that of vertices in $f(K)$.  Therefore, the vertices of $f(K)$ do not intersect $\phi(S)$ on the boundary and each component of the intersections is homeomorphic to a circle. We shall change the map $f$ up to homotopy to remove these circles to arrange $f(K)\subset \text{Star}(\phi)$.

\noindent{\bf Step 1:} We first remove all circles on $\phi(S)$ that are nullhomotopic in $M$. Since $\phi(S)$ is incompressible, any circle on $\phi(S)$ that is nullhomotopic in $M$, is in fact nullhomotopic in $\phi(S)$. Hence such circles bound a $2$-disk on $\phi(S)$. 

Among the set of circles in the intersection of $\phi(S)$ and the vertices of $f(K)$, we choose a maximal family of disjoint circles on $\phi(S)$. Since the circles in this family are disjoint, there is an innermost circle $C$. Suppose $C$ is in the intersection of $\phi(S)$ and $\phi_0(S)$ where $\phi_0=f(v)\in f(K)$. Since $\phi_0(S)$ is also incompressible, the circle $C$ bounds a $2$-disk $D'$ on $\phi_0(S)$ and a $2$-disk $D$ on $\phi(S)$. Since $M$ is irreducible, the embedded sphere $D\cup D'$ bounds a disk $B$ in $M$. 

If $\phi'\in \text{Star}(\phi_0)$, then $\phi'(S)$ cannot intersect $D$ and $D'$. The latter is clear, because $\phi''(S)$ does not even intersect $\phi_0(S)$. But if $\phi'(S)$ intersects $D$, since it is disjoint from $\partial D'=C$, their intersection would give circles inside $D$. Given that $C$ was an innermost circle, this is a contradiction. 

By pushing $D'$ across the disk $B$ toward $D$ and considering a nearby parallel copy, we obtain a vertex $\phi''\in K^{\delta}_{0}(M,\psi)$ so that $\phi''(S)$ is disjoint from $\phi_0(S)$ and $\phi'(S)$ for all $\phi'\in \text{Star}(\phi_0)$. Hence, we can find a homotopy $F\colon K\times [0,1]\to |K^{\delta}_{\bullet}(M,\psi)|$ where $F(-,0)=f,\, F(v,1)=\phi'$ and $F(-,1)$ is the same as $f$ on all vertices other than $v$. Therefore, by repeating this process, we can assume that the circles in the intersection of $\phi(S)$ and the vertices of $f(K)$ are not nullhomotopic. 

\noindent{\bf Step 2:} Now we assume that all circles in the intersection of $\phi(S)$ and the vertices in $f(K)$ are not nullhomotopic in $\phi(S)$. In the previous case to remove circles, we used embedded disks in $M$ where we thought of the disk $B$ as a ``{\it pinched product}" between two $2$-disks $D$ and $D'$. By a pinched product $P$ over  $\Sigma$ a surface with boundary, we mean the quotient of the product $\Sigma\times [0,1]$ by all segments $\{x\}\times [0,1]$ where $x\in \partial \Sigma$. This pinched product is a handlebody with corner $\partial \Sigma$ whose boundary $\partial P$ is a union of two copies of $\Sigma$ that we denote them by $\partial_{-}P$ and $\partial_{+}P$. These two copies  intersect in the corner $\partial\partial_{-}P=\partial\partial_{+}P$. To reduce the number of circles in the intersection of $\phi(S)$ and a vertex $\phi'\in K^{\delta}_{0}(M,\psi)$, we shall find a pinched product $P$ so that $\partial_{+}P$ lies on $\phi'(S)$ and $\partial_{-}P$ lies on $\phi(S)$. Then by pushing $\phi'(S)$ across $P$ and a little beyond, we obtain a new vertex $\phi''\in K^{\delta}_{0}(M,\psi)$ that has fewer circles in its intersection with $\phi(S)$. 
\begin{figure}[ht]

\begin{tikzpicture}[scale=.2]

%\begin{scope}[shift={(6,0)}]
%\draw[line width=1.05pt] [dashed] (0,.-2.5) arc (-90:90:0.5 and 2.5);
%\draw [line width=1.05pt] (0,2.5) arc (90:270:0.5 and 2.5);
%%\draw [line width=1.05pt] (0,0)+(0,-2.5) arc (-90:90:10.5 and 2.5);
%\draw [line width=1.1pt] (4.6,.5) arc (240:300:2.5 and 6.75);
%\draw   [line width=1.7pt] (6.87, 0.13) arc (48:150:1.2 and 0.9);
%
%\draw  [red, ultra thick] (-0.5,0) to [out= 20, in=140] (3.5,0);
%
%\draw [ thick] (-3.5,-0.5)--(-3.5,0.5);
%
%\end{scope}
 \draw[line width=1.05pt] [dashed] (0,.-2.5) arc (-90:90:0.5 and 2.5);
\draw [line width=1.05pt] (0,2.5) arc (90:270:0.5 and 2.5);
%\draw [line width=1.05pt] (0,0)+(0,-2.5) arc (-90:90:14.5 and 2.5);
\draw [line width=1.05pt] (0,2.5)--(13,2.5);
\draw [line width=1.05pt] (13,-2.5) arc (-90:90:2.5);
\draw [line width=1.05pt] (0,-2.5)--(13,-2.5);
\draw [line width=1.1pt] (5.1,.0) arc (250:295:2.5 and 6.75);
\draw   [line width=1.7pt] (6.87, 0.13) arc (48:150:1.2 and 0.9);

\begin{scope}[shift={(6,0)}]
\draw [line width=1.05pt] (5.1,.0) arc (250:295:2.5 and 6.75);
\draw   [line width=1.7pt] (6.87, 0.13) arc (48:150:1.2 and 0.9);
\end{scope}

\begin{scope}[shift={(0,7)}]
% \draw[line width=1.05pt] [dashed] (0,.-2.5) arc (-90:90:0.5 and 2.5);
%\draw [line width=1.05pt] (0,2.5) arc (90:270:0.5 and 2.5);
%%\draw [line width=1.05pt] (0,0)+(0,-2.5) arc (-90:90:14.5 and 2.5);
\draw [line width=1.05pt] (3,2.5)--(13,2.5);
\draw [line width=1.05pt] (13,-2.5) arc (-90:90:2.5);
\draw [line width=1.05pt] (5,-2.5)--(13,-2.5);
\draw [line width=1.1pt] (5.1,.0) arc (250:295:2.5 and 6.75);
\draw   [line width=1.7pt] (6.87, 0.13) arc (48:150:1.2 and 0.9);

\begin{scope}[shift={(6,0)}]
\draw [line width=1.05pt] (5.1,.0) arc (250:295:2.5 and 6.75);
\draw   [line width=1.7pt] (6.87, 0.13) arc (48:150:1.2 and 0.9);
\end{scope}

\end{scope}
\begin{scope}[shift={(0,-7)}]
% \draw[line width=1.05pt] [dashed] (0,.-2.5) arc (-90:90:0.5 and 2.5);
%\draw [line width=1.05pt] (0,2.5) arc (90:270:0.5 and 2.5);
%%\draw [line width=1.05pt] (0,0)+(0,-2.5) arc (-90:90:14.5 and 2.5);
\draw [line width=1.05pt,dashed] (5,2.5)--(13,2.5);
\draw [line width=1.05pt,dashed] (13,-2.5) arc (-90:90:2.5);
\draw [line width=1.05pt,dashed] (3,-2.5)--(13,-2.5);
\draw [line width=1.1pt,dashed] (5.1,.0) arc (250:295:2.5 and 6.75);
\draw   [line width=1.7pt,dashed] (6.87, 0.13) arc (48:150:1.2 and 0.9);

\begin{scope}[shift={(6,0)}]
\draw [line width=1.05pt,dashed] (5.1,.0) arc (250:295:2.5 and 6.75);
\draw   [line width=1.7pt,dashed] (6.87, 0.13) arc (48:150:1.2 and 0.9);
\end{scope}

\end{scope}
%\draw [line width=1.05pt,gray] (8,10.5) .. controls (14,14) and (12,8)  .. (11,7);
%\draw [line width=1.05pt,gray] (0,2.5)--(8,10.5);
%\draw [line width=1.05pt,gray] (0,-2.5)--(11,7);
\draw [line width=1.05pt] (15.5,7)--(15.5,0);
\draw  [line width=1.05pt] (0,2.5) to [out= 0, in=180] (3,9.5);
\draw  [line width=1.05pt] (0,-2.5) to [out= 0, in=180] (5,4.5);

\draw  [line width=1.05pt,dashed] (0,2.5) to [out= 0, in=180] (5,-4.5);
\draw  [line width=1.05pt,dashed] (0,-2.5) to [out= 0, in=180] (3,-9.5);

%\begin{scope}[shift={(-6.8,7)}, scale=0.5]
%\draw [line width=1.05pt] (4.6,.5) arc (240:300:2.5 and 6.75);
%\draw   [line width=1.7pt] (6.87, 0.13) arc (48:150:1.2 and 0.9);
%\end{scope}

%\draw  [] (-0.5,0) to [out= -20, in=-140] (5.1,.0);
%\draw  [dashed] (0.5,1) to [out= 20, in=140] (5.1,.0);
\end{tikzpicture}  
  \caption{Pushing across the pinched product}

 \end{figure}

Similar to \cite[Page 346]{MR0420620} and \cite[Step 3]{hatcher1999spaces}, we use  the covering  $p:(\widetilde{M},\tilde{x})\to (M,x)$ corresponding to the subgroup $\pi_1(\phi(S),x)$ in $\pi_1(M,x)$ for a base point $x\in \phi(S)$ to do surgery. There is a homeomorphic lift $\widetilde{\phi}(S)$ of $\phi(S)$ in $\widetilde{M}$ containing $\tilde{x}$. For each vertex $f(v)\in f(K)$, if $f(v)(S)$ intersects $\phi(S)$, we choose a lift $\widetilde{f(v)}(S)$ of the surface $f(v)(S)$ that intersects $\widetilde{\phi}(S)$, otherwise we choose any lift of $f(v)(S)$. Given our choice of covering, the map $p$ restricted to $\widetilde{f(v)}(S)$ is a homeomorphism. In this  way, we have a lift $\widetilde{f}(K)$ to $|K_{\bullet}^{\delta}(\widetilde{M};\widetilde{\phi})|$. We shall use $K_{\bullet}^{\delta}(\widetilde{M};\widetilde{\phi})$ as a bookkeeping to change $f$ inductively up to a simplicial homotopy so that its image lies in the star of $\phi(S)$. 

As Hatcher showed in \cite[Step 3]{hatcher1999spaces}, the incompressibility of $\phi(S)$ implies that every connected component of $p^{-1}(\phi(S)) $ separates $\widetilde{M}$ in two components. Let $\widetilde{S}$ be a nearby parallel copy of $\widetilde{\phi(S)}$. For each component $S_i$ of $ p^{-1}(\phi(S))$, let $M_{S_i}$ denote the component of $\widetilde{M}\backslash S_i$ that does not contain the boundary $\partial \widetilde{S}$. We order these components by inclusion. Let $M_{S_i}$ be a minimal component that intersects the union $\cup_{v\in K}\widetilde{f(v)}(S)$. 

Let $C_v$ be a component of $\widetilde{f(v)}(S)\cap M_{S_i}$. Laudenbach showed (see \cite[Corollary II.4.2]{MR0356056} and also \cite[page 8]{hatcher1999spaces}) that there is a unique pinched product $P_v$ so that $\partial_{+} P=C_v$ and $\partial_{-}P_v$ lies on $S_i$. For those $v\in K$ that $\widetilde{f(v)}(S)$ intersects $S_i$, we have a partial order on the subsurfaces $\partial_{-}P_v$ in $S_i$ given by inclusion.

To be able to use the pinched products $P_v$, to change $\widetilde{f}$ up to simplicial homotopy, we need \Cref{tech}. Given this lemma,  the rest of the argument is follows. We choose a maximal family $V$ of vertices $v \in K$ so that their corresponding subsurfaces $\partial_{-}P_v$ on $S_i$ are either disjoint or one includes the other. Let $z\in V$ be a vertex for which $\partial_{-}P_z$ is innermost. Then by moving $\partial_{+}P_z$ along $P_z$ and a little beyond, we obtain a vertex $y\in K_{0}^{\delta}(\widetilde{M};\widetilde{\phi})$ that does not intersect $S_i$ in $\partial_{-}P_z$ anymore. Note that the restriction of the covering map $p:\widetilde{M}\to M$ to the pinched product $P_z$ is a homeomorphism. Hence $p(P_z)$ is also a pinched product. 

Now we shall observe that if $w\in \text{Star}(z)$, then $f(w)\in \text{Star}(p(y))$. Because if $f(w)(S)$ intersects $p(y)(S)$, then  it has to intersect the pinched product $p(P_z)$. Therefore, its lift $ \widetilde{f(w)}(S)$ intersects $\partial_{-}P_z$ and by \Cref{tech} below, we would have $\partial_{-}P_w\subset \partial_{-}P_z$ which contradicts the fact that $\partial_{-}P_z$ was innermost. Also by considering nearby parallel copy, we can assume that $p(y)(S)$ is also disjoint from $f(z)(S)$. Therefore, $p(y)$ is connected to $f(z)$ and $f(w)$ for all $w\in \text{Star}(z)$. Hence, we can find a homotopy $F\colon K\times [0,1]\to |K_{\bullet}^{\delta}({M};\phi)|$ where $F(-,0)=f,\, F(z,1)=p(y)$ and $F(-,1)$ is the same as $f$ on all vertices other than $z$. Therefore, by repeating this process, we can reduce the number of the circles in the intersection of  the vertices of $f(K)$ and  $\phi(S)$ which would give a homotopy of $f$ to a map whose image is in $\text{Star}(\phi)$.
\end{proof}
\begin{lem}\label{tech}
 Suppose $w\in \textnormal{\text{Star}}(v)$ and $\widetilde{f(w)}(S)$  intersects $S_i$, then $\partial_{-}P_v$ and $\partial_{-}P_w$ are either disjoint or one contains the other. 
\end{lem}
\begin{proof}
Suppose the contrary, so the intersection of  $\partial_{-}P_v$ and $\partial_{-}P_w$ is a subsurface $\Sigma\subset S_i$ such that its boundary $\partial \Sigma$ decomposes into two parts $\partial_{v}\Sigma\subset \partial \partial_{-}P_v$ and $\partial_{w}\Sigma\subset \partial \partial_{-}P_w$. Over $\Sigma$ in $P_v$, we shall choose   a ``{\it partial}" pinched product $Q_v$  which is a submanifold in $P_v$ with corner (see Figure \ref{ppinch}). The manifold $Q_v$ is homeomorphic to the quotient of $\Sigma\times[0,1]$ given by pinching $\{x\}\times [0,1]$ for all $x\in \partial_{v}\Sigma$. The boundary of the partial pinched product $Q_v$ is the union of $\partial_{-}Q_v=\Sigma$, a piece $\partial_0 Q_v$ that is homeomorphic to $\partial_w\Sigma\times [0,1]$ and $\partial_{+}Q_v$ that lies on $\partial_{+}P_v$. 
\begin{figure}[ht]\label{ppinch}

\begin{tikzpicture}[scale=.2]

 \draw[line width=1.05pt] [dashed] (0,.-2.5) arc (-90:90:0.5 and 2.5);
\draw [line width=1.05pt] (0,2.5) arc (90:270:0.5 and 2.5);
%\draw [line width=1.05pt] (0,0)+(0,-2.5) arc (-90:90:14.5 and 2.5);
\draw [line width=1.05pt] (0,2.5)--(23,2.5);
\draw [line width=1.05pt] (0,2.5)--(-4,2.5);
\draw [line width=1.05pt, dashed] (-8,2.5)--(-4,2.5);
\draw [line width=1.05pt, dashed] (25,2.5)--(23,2.5);

\draw [line width=1.05pt] (1,20)--(15,20);
\draw [line width=1.05pt] (1,15)--(15,15);

\draw [line width=1.05pt] (0,-2.5)--(23,-2.5);
\draw [line width=1.05pt] (0,-2.5)--(-4,-2.5);
\draw [line width=1.05pt, dashed] (-8,-2.5)--(-4,-2.5);
\draw [line width=1.05pt, dashed] (25,-2.5)--(23,-2.5);

\draw [line width=1.05pt] (1,20) arc (90:270:2.5);

\draw [line width=1.1pt] (5.1,.0) arc (250:295:2.5 and 6.75);
\draw   [line width=1.7pt] (6.87, 0.13) arc (48:150:1.2 and 0.9);

\begin{scope}[shift={(6,17.5)}]
\draw [line width=1.05pt] (5.1,.0) arc (250:295:2.5 and 6.75);
\draw   [line width=1.7pt] (6.87, 0.13) arc (48:150:1.2 and 0.9);
\end{scope}

\begin{scope}[shift={(0.5,17.5)}]
\draw [line width=1.05pt] (5.1,.0) arc (250:295:2.5 and 6.75);
\draw   [line width=1.7pt] (6.87, 0.13) arc (48:150:1.2 and 0.9);
\end{scope}

\begin{scope}[shift={(-3,17.5)}]
\draw [line width=1.05pt] (5.1,.0) arc (250:295:2.5 and 6.75);
\draw   [line width=1.7pt] (6.87, 0.13) arc (48:150:1.2 and 0.9);
\end{scope}

\begin{scope}[shift={(19,0)}]
 \draw[line width=1.05pt] [dashed] (0,.-2.5) arc (-90:90:0.5 and 2.5);
\draw [line width=1.05pt] (0,2.5) arc (90:270:0.5 and 2.5);
\end{scope}

\begin{scope}[shift={(9,7)}]
 \draw[line width=1.05pt] [dashed] (0,.-2.5) arc (-90:90:0.5 and 2.5);
\draw [line width=1.05pt] (0,2.5) arc (90:270:0.5 and 2.5);
\end{scope}

\begin{scope}[shift={(9,17.5)}]
 \draw[line width=1.05pt] [dashed] (0,.-2.5) arc (-90:90:0.5 and 2.5);
\draw [line width=1.05pt] (0,2.5) arc (90:270:0.5 and 2.5);
\end{scope}

\begin{scope}[shift={(16,0)}]
\draw [line width=1.05pt] (5.1,.0) arc (250:295:2.5 and 6.75);
\draw   [line width=1.7pt] (6.87, 0.13) arc (48:150:1.2 and 0.9);
\end{scope}

\begin{scope}[shift={(-9,0)}]
\draw [line width=1.05pt] (5.1,.0) arc (250:295:2.5 and 6.75);
\draw   [line width=1.7pt] (6.87, 0.13) arc (48:150:1.2 and 0.9);
\end{scope}
\begin{scope}[shift={(0,7)}]

\draw [line width=1.05pt] (3,2.5)--(13,2.5);
\draw [line width=1.05pt] (13,-2.5) arc (-90:90:2.5);
\draw [line width=1.05pt] (5,-2.5)--(13,-2.5);
\draw [line width=1.1pt] (5.1,.0) arc (250:295:2.5 and 6.75);
\draw   [line width=1.7pt] (6.87, 0.13) arc (48:150:1.2 and 0.9);

\begin{scope}[shift={(6,0)}]
\draw [line width=1.05pt] (5.1,.0) arc (250:295:2.5 and 6.75);
\draw   [line width=1.7pt] (6.87, 0.13) arc (48:150:1.2 and 0.9);
\end{scope}

\end{scope}
\begin{scope}[shift={(0,-7)}]

\end{scope}

\draw  [line width=1.05pt] (0,2.5) to [out= 0, in=180] (3,9.5);
\draw  [line width=1.05pt] (0,-2.5) to [out= 0, in=180] (5,4.5);

\draw  [line width=1.05pt] (19,-2.5) to [out= 140, in=0] (15,15);
\draw  [line width=1.05pt] (19,2.5) to [out= 110, in=0] (15,20);

\draw  [line width=1.05pt, dotted] (19,-2.5) to [out= 175, in=2] (9,4.5);
\draw  [line width=1.05pt, dotted] (19,2.5) to [out= 195, in=-60] (9,9.5);

\node at (7.5, 11) {$\partial_{+}P_v$};
\node at (20.5, 20) {$\partial_{+}P_w$};

\node at (9.5, -4) {$\Sigma$};
\node at (0, -4) {$\partial_v\Sigma$};
\node at (19.5, -4) {$\partial_w\Sigma$};

\end{tikzpicture}  
  \caption{The dotted line is $\partial_0 Q_v$}
   \label{ppinch}

 \end{figure}

We shall  change $\partial_0 Q_v$ up to isotopy to make it disjoint from $P_w$.  Since $w\in \textnormal{\text{Star}}(v)$, we know $\partial_{+}P_v$ and $\partial_{+}P_w$ are disjoint. Therefore, $\partial_{+}Q_v$ either lies inside or outside of $P_w$. First, let us assume that it is inside of $P_w$. Each connected component $L_i$ of $\partial_0 Q_v$  is homeomorphic to a cylinder. We first change $\partial_0 Q_v$ relative to its boundary up to isotopy to make it transverse to the surface $\partial_{+}P_w$. The intersection on $L_i$ are circles. If a circle is nullhomotopic on $L_i$, by irreducibility of $\widetilde{M}$ and incompressibility of $\partial_{+}P_w$, we can remove it by an isotopy. So we assume all the circles in the  intersection $L_i\cap \partial_{+}P_w$ are disjoint isotopic circles on the cylinder $L_i$. 

Let  $L_i^{\alpha}$ be a piece of the cylinder $L_i$ that lies outside of $P_w$ and bounds two consecutive circles $C^{\alpha}_1$ and $C^{\alpha}_2$ on $L_i$. Since $C^{\alpha}_1$ and $C^{\alpha}_2$ are also isotopic on $\partial_{+}P_w$, they bound a cylinder $L_w^{\alpha}$ on $\partial_{+}P_w$. 
Hence, by the Laudenbach theorem (\cite[Corollary II.4.2]{MR0356056}) again these two cylinders are isotopic. So we can push  $\partial_0 Q_v$ by an isotopy inside $P_w$. Therefore, the surface $\partial_{+}P_v$ lies inside $P_w$ but this contradicts the incompressibility of $\partial_{+}P_v$. Because the fundamental group of $P_v$ and $P_w$ are the same as that of $\partial_{-}P_v$ and $\partial_{-}P_w$. Since the surfaces $\partial_{-}P_v$ and $\partial_{-}P_w$ are not nested, incompressibility implies that the fundamental groups $P_v$ and $P_w$ as subgroups of $\pi_1(\widetilde{M},\tilde{x})$ are not nested either. Hence, the surface $\partial_{+}P_v$ cannot lie inside $P_w$ which is a contradiction. Similarly if $\partial_{+}Q_v$ lies outside, we arrive at a contradiction by showing that $\partial_{+}P_w$ lies inside $P_v$.
\end{proof}
Using \Cref{cutincomp}, we can resolve $\BdH_{0,\partial}(M)$ and $\mathrm{B}|S_{\bullet}(\tH_{0,\partial}(M))|$ and we have a natural map between them. Therefore, exactly the same argument as \Cref{reduction2}, implies that \Cref{main} holds if it holds for $M\backslash \sigma_p$ for all $\sigma_p\in K^{\delta}_p(M;\psi)$. So, now we can do an inductive argument to prove the Thurston theorem for Haken manifolds with boundary.
\begin{thm}
\Cref{main} holds for Haken manifolds with boundary.
\end{thm}
\begin{proof}
First we assume that $M$ is a handlebody and we induct on the genus. The base case is the disk which is the Mather theorem. We choose an incompressible disk $\psi\colon D^2\to M$. Note that for a $p$-simplex $\sigma_p\in K_p^{\delta}(M;\psi)$, the manifold $M\backslash \sigma_p$ is homeomorphic to the disjoint union of $p$ disks and a handlebody of lower genus. So by induction Thurston's theorem holds for $M\backslash \sigma_p$ for all $p$ and all $p$-simplices $\sigma_p$. Therefore, it also holds for $M$. 

Now for a general Haken manifold $M$, we have a finite Haken hierarchy i.e. there are finite number of incompressible surfaces in $M$ such that if we cut $M$ along those surfaces, we obtain disjoint union of disks. We shall induct on the Haken hierarchy. For an incompressible surface $\psi\colon S\to M$ and a $p$-simplex $\sigma_p\in  K_p^{\delta}(M;\psi)$, the manifold $M\backslash \sigma_p$ is homeomorphic to the disjoint union of $p$ handlebodies, each homeomorphic to $S\times [0,1]$, and the manifold $M\backslash \psi$ which is lower in the hierarchy. Therefore, by induction and the previous case, Thurston's theorem holds for all $M\backslash \sigma_p$. Hence, it also holds for $M$. 
\end{proof}
\section{On the homotopy type of $\tH_{0,\partial}(M)$}\label{sec4} In this section, we use some of the semisimplicial resolutions in previous sections to give new proofs of the contractibility of $\tH_{0,\partial}(M)$ when $M$ is a hyperbolic surface (\cite{hamstrom1974homotopy}) or when it is a Haken $3$-manifold with boundary (\cite{hatcher1999spaces, MR0420620}). 

Our strategy is to show that in these cases the classifying space  $\BH_{0,\partial}(M)$ is acyclic. Since it is simply connected, Whitehead's theorem implies that it should be weakly contractible. Therefore,  the weak contractibility of $\tH_{0,\partial}(M)$ follows from the weak contractibility of its delooping. 

The statements hold for diffeomorphism groups of these manifolds but since we defined resolutions for homeomorphism groups we give the argument for homeomorphism groups. Also another convenience of working with homeomorphism groups is that we can use the Thurston theorem \ref{main} that the map 
\[
\mathrm{B}\tdH_{0,\partial}(M)\to \BH_{0,\partial}(M),
\]
induces a homology isomorphism. Therefore, if we want to show that $\tH_{0,\partial}(M)$ is weakly contractible, it is enough to show that $\tdH_{0,\partial}(M)$ is an acyclic group. 
\subsection{The case of hyperbolic surfaces} In this section we prove
\begin{thm}\label{surface}The group $\tdH_{0,\partial}(\Sigma)$ is an acyclic group when $\Sigma$ is a hyperbolic surface.
\end{thm}
\begin{proof} We consider the case of closed surfaces and surfaces with boundary separately. 

\noindent{\bf Case 1:} Suppose $\Sigma$ has a nonempty boundary. We induct on $-\chi(\Sigma)$. The base case is when $\Sigma$ is a disk which is the Mather theorem (\cite{MR0288777}). If the genus $g(\Sigma)$ is not zero, we choose an arc $\phi$ whose two ends lie on a same boundary component so that cutting along $\phi$ decreases the genus. If the genus is zero, we choose an arc $\phi$ whose two ends lie on different boundary components so that cutting along $\phi$ decreases the number of boundary components.  In either case, the realization of $B_{\bullet}^{\delta}(\Sigma,\phi)$ is contractible (see \Cref{arc}) which gives a semisimplicial resolution 
\[
Y_{\bullet}^{\delta}(\Sigma,\phi)\to \mathrm{B}\tdH_{0,\partial}(\Sigma).
\]
Therefore, as we discussed in \Cref{reduction2}, we have a spectral sequence
\[
E^1_{p,q}=H_q(Y_{p}^{\delta}(\Sigma,\phi))\Rightarrow H_{p+q}(\mathrm{B}\tdH_{0,\partial}(\Sigma)).
\]
The homotopy type of $Y_{p}^{\delta}(\Sigma,\phi)$ is the same as $\coprod_{\text{orbits}}\mathrm{B}\text{\textnormal{Stab}}^{\delta}(\sigma_p)$ where $\sigma_p$  varies over the set of the orbits of the action of $\tdH_{0,\partial}(\Sigma)$ on $B_{p}^{\delta}(\Sigma,\phi)$. 

\noindent{\bf Claim:}  For all $\sigma_p$, the group $\text{\textnormal{Stab}}^{\delta}(\sigma_p)$ is an acyclic group. 

\noindent {\it Proof of the claim:}
Let $\Sigma\backslash \sigma_p$ denote a surface obtained from $\Sigma$ by cutting along the arcs in $\sigma_p$. This surface is homeomorphic to the disjoint union of $\Sigma\backslash \phi$ with $p$ disjoint disks. Hence using the induction hypothesis and the Mather theorem for disks, the group $\tdH_{0,\partial}(\Sigma\backslash \sigma_p)$ is acyclic. Recall from diagram (\ref{ee}), we have a fibration 
\[
\mathrm{B}\tdH_{0,\partial}(\Sigma\backslash \sigma_p)\to \mathrm{B}\text{\textnormal{Stab}}^{\delta}(\sigma_p)\to \mathrm{B}\pi_0(\text{\textnormal{Stab}}(\sigma_p)).
\]
Therefore, if we show that $\pi_0(\text{\textnormal{Stab}}(\sigma_p))$ is trivial, we can conclude that $\text{\textnormal{Stab}}^{\delta}(\sigma_p)$ is an acyclic group. Note that $\pi_0(\text{\textnormal{Stab}}(\sigma_p))$ is the kernel of the map
\[
\pi_0(\tH_{0,\partial}(\Sigma\backslash \sigma_p))\to \pi_0(\tH_{0,\partial}(\Sigma)). 
\]
But this kernel is trivial (see \cite[Corollary 4.2]{MR1752295} for an elementary proof) i.e. if $f\in \tH_{0,\partial}(\Sigma)$ fixes $\sigma$, then it is isotopic to the identity relative the arcs in $\Sigma$. Hence $H_q(\mathrm{B}\text{\textnormal{Stab}}^{\delta}(\sigma_p))=0$ unless $q=0$ in which case it is isomorphic to $\bZ$. $\blacksquare$

On the other hand two $p$-simplices $\sigma=(\phi_0,\phi_1,\dots,\phi_p)$ and $\sigma'=(\phi'_0,\phi'_1,\dots,\phi'_p)$ are on the same orbit if and only if the corresponding $t$-coordinates $t_{\phi_i}=t_{\phi'_i}$ are the same for all $i$. Therefore, the set of orbits of the action of $\tdH_{0,\partial}(\Sigma)$ on $B_p(\Sigma, \phi)$ is the same as  $\text{Conf}_{p+1}(\bR)$ the set configurations of $p+1$ points on the real line.

Hence, $E^1$-page is concentrated in the first line when $q=0$ and $E^1_{0,p}=\bZ[\text{Conf}_{p+1}(\bR)]$. But the chain complex $(\bZ[\text{Conf}_{p+1}(\bR)], d_1)$ calculates the homology of $|\text{Conf}_{\bullet+1}(\bR)|$ which is an infinite simplex on $\bR$. Therefore, the spectral sequence converges to zero in positive degrees which implies that $\tdH_{0,\partial}(\Sigma)$ is an acyclic group. 

\noindent {\bf Case 2:} Suppose $\Sigma$ is a closed surface. Reducing the case of closed surfaces to the case of surfaces with boundary could be done using the long exact sequence for the fibration $\tH(\Sigma)\to \text{Emb}(D^2,\Sigma)$ and Birman's exact sequence. But here, we give an argument using the $0$-handle resolution which might be useful for closed hyperbolic $3$-manifolds as we shall discuss in \Cref{Gabai}. 

In this case we use the semisimplicial set $A_{\bullet}^{\delta}(\Sigma)$ (see \Cref{defn2}) and the resolution 
\[
X_{\bullet}^{\delta}(\Sigma)\to  \mathrm{B}\tdH_{0}(\Sigma).
\]
Let $e_p\in A^{\bf t}_p(M)$ be an embedding of $p+1$ disjoint disks and let $[e_p]\in A_p^{\delta}(M)$ denote its germ. In \Cref{lem1}, we showed that $X_p^{\delta}(\Sigma)$ has the same homotopy type as $ \mathrm{B}\text{\textnormal{Stab}}^{\delta}([e_p])$. But given \Cref{lem3}, the spectral sequence for the semisimplicial resolution $X_{\bullet}^{\delta}(\Sigma)$ can be written as
\begin{equation}\label{s1}
E^1_{p,q}=H_q(\mathrm{B}\text{\textnormal{Stab}}^{\delta}(e_p))\Rightarrow H_{p+q}(\mathrm{B}\tdH_{0}(\Sigma)).
\end{equation}
Again we have a fibration 
\[
\mathrm{B}\tdH_{0,\partial}(\Sigma\backslash e_p)\to \mathrm{B}\text{\textnormal{Stab}}^{\delta}(e_p)\to \mathrm{B}\pi_0(\text{\textnormal{Stab}}(e_p)),
\]
where the group $\tdH_{0,\partial}(\Sigma\backslash e_p)$ is acyclic by the case 1. Therefore, we have
\begin{equation}\label{er}
\mathrm{B}\text{\textnormal{Stab}}^{\delta}(e_p)\to \mathrm{B}\pi_0(\text{\textnormal{Stab}}(e_p)),
\end{equation}
is a homology isomorphism. 

On the other hand, the group $\pi_0(\text{\textnormal{Stab}}(e_p))$ is the kernel of 
\[
\pi_0(\tH_{0,\partial}(\Sigma\backslash e_p))\to \pi_0(\tH_{0,\partial}(\Sigma)). 
\]
To determine $\pi_0(\text{\textnormal{Stab}}(e_p))$, consider the Kan fibration  given by the parametrized isotopy extension and \Cref{emb}
\[
S_{\bullet}((\tH_{0,\partial}(\Sigma\backslash e_p))\to S_{\bullet}(\tH_{0,\partial}(\Sigma))\to S_{\bullet}(A^{\bf t}_p(\Sigma)).
\]
The long exact sequence of homotopy groups for this fibration implies that we have a short exact sequence 
\begin{equation}\label{short}
\pi_1(A^{\bf t}_p(\Sigma))\to\pi_0(\tH_{0,\partial}(\Sigma\backslash e_p))\to \pi_0(\tH_{0}(\Sigma)). 
\end{equation}

\noindent {\bf Claim:} $\pi_1(A^{\bf t}_p(\Sigma))\cong \pi_0(\text{\textnormal{Stab}}(e_p))$

\noindent {\it Proof of the claim:} Given the exact sequence (\ref{short}), we only need to show that the  map $\pi_1(A^{\bf t}_p(\Sigma))\to \pi_0(\tH_{0,\partial}(\Sigma\backslash e_p))$ is injective. The group $\pi_1(A^{\bf t}_p(\Sigma))$ is known as the framed pure surface braid group of $\Sigma$. The pure surface braid group (not framed) is the fundamental group of the space of ordered configurations of points $\text{Conf}_{p+1}(\Sigma)$. These groups sit in a short exact sequence
\[
1\to \bZ^{p+1}\to \pi_1(A^{\bf t}_p(\Sigma))\to \pi_1(\text{Conf}_{p+1}(\Sigma))\to 1, 
\]
where $\bZ^{p+1}$ is the group generated by the Dehn twists around the boundary components in $\Sigma\backslash e_p$.  These twist and all their nonzero powers induce nontrivial inner automorphisms of the fundamental group $\Sigma\backslash e_p$ which is a free group. Therefore,  $\bZ^{p+1}$ as a subgroup of $\pi_1(A^{\bf t}_p(\Sigma))$ injects into $\pi_0(\tH_{0,\partial}(\Sigma\backslash e_p))$ and in fact we have a short exact sequence
\[
1\to \bZ^{p+1}\to\pi_0(\tH_{0,\partial}(\Sigma\backslash e_p))\to \pi_0(\tH_{0}(\Sigma,  c(e_p)))\to 1,
\]
where $\pi_0(\tH_{0}(\Sigma, c(e_p)))$ is the mapping class group of $\Sigma$ with marked points, $c(e_p)$, the centers of the disks in $e_p$. Therefore, to show that $$\pi_1(A^{\bf t}_p(\Sigma))\to\pi_0(\tH_{0,\partial}(\Sigma\backslash e_p)),$$ is injective, it is enough to show the natural map 
\begin{equation}\label{conf}
\pi_1(\text{Conf}_{p+1}(\Sigma))\to \pi_0(\tH_{0}(\Sigma, c(e_p))),
\end{equation}
is injective. This map is induced by the long exact sequence of homotopy groups for the fibration (see also \cite[Section 4]{mccarty1963homeotopy}) given by an evaluation map
\[
\tH_0(\Sigma)\to \text{Conf}_{p+1}(\Sigma).
\]
It is known that $\pi_1(\text{Conf}_{p+1}(\Sigma))$ does not have a center for hyperbolic surfaces (see \cite[Proposition 1.6]{paris1999geometric}). Therefore, the composite of the maps
\[
\pi_1(\text{Conf}_{p+1}(\Sigma))\to \pi_0(\tH_{0}(\Sigma, c(e_p)))\to \text{Aut}(\pi_1(\text{Conf}_{p+1}(\Sigma))),
\]
which is given by inner automorphisms, is injective.  Therefore, the map in (\ref{conf}) is injective. $\blacksquare$

Since a hyperbolic surface and its frame bundle are aspherical, one can use the fibration $A^{\bf t}_p(\Sigma)\to A^{\bf t}_{p-1}(\Sigma)$ inductively to conclude that $A^{\bf t}_p(\Sigma)$ is also aspherical. Therefore, we have 
\[
A^{\bf t}_p(\Sigma)\xrightarrow{\simeq} \mathrm{B}\pi_1(A^{\bf t}_p(\Sigma))\xrightarrow{\simeq}\mathrm{B}\pi_0(\text{\textnormal{Stab}}(e_p)).
\]
Given that the map (\ref{er}) is a homology isomorphism, $A^{\bf t}_p(\Sigma)$ and $\mathrm{B}\text{\textnormal{Stab}}^{\delta}(e_p)$ have the same homology groups. We have a zig-zag of semisimplicial maps 
\[
X_{\bullet}^{\delta}(\Sigma)\leftarrow A^{\bf t, \delta}_{\bullet}(\Sigma)\hcoker \tdH_{0}(\Sigma)\to A^{\bf t}_{\bullet}(\Sigma)\hcoker \tH_{0}(\Sigma)\leftarrow A^{\bf t}_{\bullet}(\Sigma),
\]
where the first two maps induce isomorphisms on $E^1$-pages (see \Cref{redbdry}). Hence, we obtain a comparison map of of the spectral sequence (\ref{s1}) with 
\[
E^1_{p,q}=H_q(A^{\bf t}_p(\Sigma))\Rightarrow H_{p+q}(|A^{\bf t}_{\bullet}(\Sigma)|),
\]
which is an isomorphism on the $E^1$-page. Since $|A^{\bf t}_{\bullet}(\Sigma)|$ is weakly contractible (see \Cref{claim}), the acyclity of the space $\mathrm{B}\tdH_{0}(\Sigma)$ follows. 
\end{proof}
\subsection{The case of Haken $3$-manifolds with boundary}Hatcher computed the homotopy type of the group of homemorphisms of Haken manifolds in \cite{MR0420620}. Given his proof of Smale's conjecture, his computation of  homeomorphisms carries over  to diffeomorphisms of the Haken manifolds (\cite{hatcher1999spaces}).  Hatcher developed a subtle disjunction technique to study embedding spaces (\cite{hatcher1983proof, MR0420620, MR624946}). In the case of  Haken $3$-manifolds, he improved Laudenbach's surgery techniques (\cite[Chapter 2.5]{MR0356056}) to do a parametrized surgery on the space of incompressible surfaces. Here, we also use Laudenbach's results, but instead of Hatcher's disjunction argument, we prove that $\mathrm{B}\tdH_{0,\partial}(M)$ is acyclic for Haken $3$-manifolds with boundary similar to the case of surfaces.

\begin{thm}Let $M$ be Haken $3$-manifold with non-empty boundary. The space $\BdH_{0,\partial}(M)$ is acyclic.
\end{thm}
\begin{proof}
We induct on the Haken hierarchy. We know that there are finite number of incompressible surfaces that if we cut along those surfaces, we obtain disjoint union of disks. Let $\psi:S\hookrightarrow M$ be an incompressible surface in $M$. Since $K^{\delta}_{\bullet}(M,\psi)$ has a contractible realization (see \Cref{cutincomp}), similar to the case of surfaces with boundary, we obtain a spectral sequence
\[
E^1_{p,q}=H_q(\coprod_{\text{orbits}}\mathrm{B}\text{Stab}^{\delta}(\sigma_p))\Rightarrow H_{p+q}(\BdH_{0,\partial}(M)),
\]
where $\sigma_p$ varies over a representative of the set of orbits of the action of $\tdH_{0,\partial}(M)$ on $K^{\delta}_{p}(M,\psi)$. First note that the orbit of $\sigma_p$ is completely determined by the $t$-coordinates of the surfaces in $\sigma_p$. Therefore, the set of orbits of the action of $\tdH_{0,\partial}(M)$ on $K^{\delta}_{p}(M,\psi)$ is identified with $\text{Conf}_{p+1}(\bR)$. 

\noindent {\bf Claim:} $\text{Stab}^{\delta}(\sigma_p)\cong \tdH_{0,\partial}(M\backslash \sigma_p)$.

This is a consequence of the Laudenbach result (\cite[pp. 48-62]{MR0356056}). In other words, we need to show that $\pi_0(\text{Stab}(\sigma_p))$ is trivial. Laudenbach showed that if $f\in \tH_{0,\partial}(M)$ fixes an incompressible surface $S$, then $f$ is isotopic to the identity relative to $S$ which implies that $\pi_0(\text{Stab}(\sigma_p))$ is trivial for all $p$.

On the other hand, note that $M\backslash \sigma_p$ is homeomorphic to the disjoint union of $M\backslash \psi$ with $p$ disjoint handlebodies homeomorphic to $S\times [0,1]$. By the induction hypothesis, we know that $\tdH_{0,\partial}(M\backslash \psi)$ is an acyclic group. For the handlebodies we can induct on the genus and do exactly the same argument to cut along incompressible disks to deduce that $\tdH_{0,\partial}(S\times [0,1])$ is also acyclic. Hence, $\text{Stab}^{\delta}(\sigma_p)$ is an acyclic group. So the $E^1$-page is concentrated in the first line when $q=0$ and is isomorphic to the chain complex $(\bZ[\text{Conf}_{\bullet+1}(\bR)], d_1)$ which calculates the homology of an infinite simplex. Therefore, the spectral sequence converges to zero in positive degrees. 
\end{proof}
\subsection{Further discussion}\label{Gabai}We end with a question about hyperbolic three manifolds. Let $M$ be closed hyperbolic $3$-manifold. Gabai in \cite{gabai2001smale} used his high powered ``insulator" machinary (see \cite{gabai1997geometric}) and minimal surface theory to prove that $\tH_0(M)$ is contractible by reducing to the case of Haken manifolds with boundary. It would be interesting if the techniques of this paper could prove Gabai's theorem without using high powered tools in geometry. It is desirable to have an argument similar to the case of closed surfaces (Case 2 in the proof of \Cref{surface}). In that case, we used a semisimplicial resolution to reduce to the case of surfaces with boundary. Surfaces with boundary behave like Haken $3$-manifolds. It would be interesting to define a semisimplicial resolution by cutting certain submanifolds, like solid tori, to reduce to the case of Haken $3$-manifolds with boundary.  One candidate for  a semisimplicial resolution for $\BH_0(M)$ could be as follows. Let  $\gamma$ be a closed geodesic in $M$. Fix a parametrized tubular neighborhood of $\gamma$ by embedding $\phi: D^2\times S^1\hookrightarrow M$ so that $\phi(\{(0,0)\}\times S^1)=\gamma$. 
\begin{defn}Let $C_{\bullet}(M)$ be a semisimplicial set whose set of $0$ simplices is given by  oriented closed curves that are isotopic to $\gamma$. We define $C_p(M)$ as a subset of $C_0(M)^{p+1}$ to be the set of $(p+1)$-tuples $\sigma_p=(\gamma_0,\gamma_1,\dots, \gamma_p)$ so that there exists a homeomorphism $f_{\sigma_p}\in \tH_0(M)$ where $f_{\sigma_p}(\gamma_i)=\phi(\{(t_i,0)\}\times S^1)$ for a $t_i$ such that $ t_0<t_1<\dots<t_p$. The $i$-th face maps is given by forgetting the $i$-th curve.
\end{defn}
\begin{quest}
Is $|C_{\bullet}(M)|$ weakly contractible?
\end{quest}
We need  to show that realization of the semi-simplicial set $C_{\bullet}(M)$ is contractible. If the answer to this question is affirmative, one could give a simpler proof of Gabai's theorem as follows: Consider the semisimplicial resolution
\[
C_{\bullet}(M)\hcoker \tdH_0(M)\to \BdH_0(M).
\]
Since the action of $\tdH_0(M)$ on $C_{\bullet}(M)$ is transitive, for a $p$-simplex $\sigma_p$ in $C_p(M)$, we have $C_p(M)\hcoker \tdH_0(M)\simeq \mathrm{B}\text{Stab}(\sigma_p)$. Given that the complement of $\sigma_p$ in $M$ is a Haken manifold, the identity component of $\text{Stab}(\sigma_p)$ is contractible, therefore $ \mathrm{B}\text{Stab}(\sigma_p)\simeq  \mathrm{B}\pi_0(\text{Stab}(\sigma_p))$. On the other hand, using JSJ decomposition and some hyperbolic geometry, it is not hard to show that $\pi_0(\text{Stab}(\sigma_p))$ is isomorphic to the pure braid group $\text{PBr}_{p+1}$. Hence, one might have a spectral sequence
\[
E^1_{p,q}=H_q(\mathrm{B}\text{PBr}_{p+1})\Rightarrow H_{p+q}(\BdH_0(M);\bZ),
\]
but recall that a model for $\mathrm{B}\text{PBr}_{p+1}$ is an ordered configuration space $\text{Emb}([p], D^2)$. Thus the above spectral sequence converges to the realization of the semi simplicial space $\text{Emb}([\bullet], D^2)$. Now from \Cref{claim1} we know that the realization of $\text{Emb}([\bullet], D^2)$ is weakly contractible, therefore the above spectral sequence converges to zero in positive degrees.  

 \bibliographystyle{alpha}
\bibliography{reference}
\end{document}